\documentclass[11pt]{amsart}
\newtheorem{theorem}{Theorem}[section]
\newtheorem{proposition}[theorem]{Proposition}
\newtheorem{lemma}[theorem]{Lemma}
\newtheorem{corollary}[theorem]{Corollary}
\newtheorem{remark}[theorem]{Remark}
\usepackage{comment}

\newtheorem{question}[theorem]{Question}
\theoremstyle{definition}
\newtheorem{definition}[theorem]{Definition}
\newtheorem{claim}[theorem]{Claim}
\usepackage{hyperref}
\usepackage{ amssymb }
\usepackage{color}
\newcommand{\eps}{\varepsilon}

\usepackage{mathrsfs}
\usepackage{dsfont}
\usepackage{txfonts}
\usepackage{mathabx}

\usepackage{graphicx}
\usepackage[all]{xy}
\usepackage{pgf,tikz,pgfplots}
\usepackage{subfigure}
\usepackage{tikz-3dplot}
\usetikzlibrary{patterns,arrows}
\tikzset{
    partial ellipse/.style args={#1:#2:#3}{
        insert path={+ (#1:#3) arc (#1:#2:#3)}
    }
}

\begin{document}

\title[A nonexistence result for wing-like mean curvature flows]{A nonexistence result for wing-like mean curvature flows in $\mathbb{R}^4$}

\author{Kyeongsu Choi, Robert Haslhofer, Or Hershkovits}

\begin{abstract}
Some of the most worrisome potential singularity models for the mean curvature flow of $3$-dimensional hypersurfaces in $\mathbb{R}^4$ are  noncollapsed wing-like flows, i.e. noncollapsed flows that are asymptotic to a wedge. In this paper, we rule out this potential scenario, not just among self-similarly translating singularity models, but in fact among all ancient noncollapsed flows in $\mathbb{R}^4$.
Specifically, we prove that for any ancient noncollapsed mean curvature flow $M_t=\partial K_t$ in $\mathbb{R}^4$ the blowdown $\lim_{\lambda\to 0} \lambda\cdot {K_{t_0}}$ is always a point, halfline, line, halfplane, plane or hyperplane, but never a wedge. In our proof we introduce a fine bubble-sheet analysis, which generalizes the fine neck analysis that has played a major role in many recent papers. Our result is also a key first step towards the classification of ancient noncollapsed flows in $\mathbb{R}^4$, which we will address in a series of subsequent papers.
\end{abstract}

\maketitle

\tableofcontents

\section{Introduction}

The mean curvature flow of mean-convex surfaces in $\mathbb{R}^3$, and also of $2$-convex hypersurfaces in $\mathbb{R}^{n+1}$, is by now well understood.  In particular, White proved that the evolving surfaces are smooth away from a small set of times \cite{White_size,White_nature}, and Huisken-Sinestrari \cite{HuiskenSinestrari_surgery}, Brendle-Huisken \cite{BH_surgery} and Haslhofer-Kleiner \cite{HaslhoferKleiner_surgery} constructed a flow with surgery. More recently, in significant work by Angenent-Daskalopoulos-Sesum \cite{ADS,ADS2} and Brendle-Choi \cite{BC,BC2} a complete classification of all possible singularity models for such flows has been obtained (these classification results have in turn be generalized in our recent proof of the mean-convex neighborhood conjecture \cite{CHH,CHHW}).\\

Some important results are of course valid in arbitrary dimensions. In particular, tangent flows, i.e. blowup limits centered at some fixed space-time point, are always self-similarly shrinking thanks to Huisken's monotonicity formula \cite{Huisken_monotonicity}. By fundamental results of Huisken \cite{Huisken_shrinker} and Colding-Minicozzi \cite{CM_generic}  the only mean-convex shrinkers (and also the only stable shrinkers) are round spheres and cylinders. This of course is an important ingredient in the above quoted regularity theory for mean-convex flows due to White. More recently, Colding-Minicozzi proved uniqueness of cylindrical tangent flows \cite{CM_uniqueness}, and applied this to obtain refined information about the structure of the singular set in flows that only encounter stable singularities \cite{CM_singular_set}. 
On the other hand, general blowup limits, i.e. blowup limits along arbitrary sequences of space-time points, are much less understood already for the flow of $3$-dimensional hypersurfaces in $\mathbb{R}^4$. In particular, one of the most worrisome potential scenarios is that one encounters so-called wing-like flows as blowup limits, i.e. flows that are asymptotic to a wedge.\\

To explain the background, let us first discuss the situation of $2$-dimensional flows in $\mathbb{R}^3$, and in fact let us restrict the discussion even further to the case of strictly convex, graphical, self-similarly translating solutions. The most well known such solution is the translating bowl soliton \cite{AltschulerWu}, which is given by a rotationally symmetric entire graph. However, in addition to the bowl there is also a one-parameter family of solutions defined over strip regions $(-b,b)\times\mathbb{R}$ for every $b>\pi/2$. These solutions have been constructed by Ilmanen (unpublished) and Wang \cite{Wang_convex}, and have been classified in recent work by Hoffman-Ilmanen-Martin-White \cite{HIMW}, building also on important prior work by Spruck-Xiao \cite{SpruckXiao}. Ilmanen called them $\Delta$-wings, capturing their shape. They are asymptotic to a wedge with its sides modelled on the grim-reaper times $\mathbb{R}$.\\

The $2$-dimensional $\Delta$-wings do not cause any deep concerns for the singularity analysis of mean curvature flow. This is because they are collapsed. In fact, it is known that collapsed solutions cannot arise as blowup limit of any mean-convex flow \cite{White_size,White_nature} (see also \cite{ShengWang,Andrews_noncollapsing,HaslhoferKleiner_meanconvex}), and conjectured that they cannot arise as blowup limit of any embedded flow \cite{Ilmanen_monotonicity}. We recall that a mean-convex flow is called noncollapsed, if there is some $\alpha>0$ such that every point admits interior and exterior balls of radius at least $\alpha/H(p)$.\\

However, the situation changes dramatically when one moves one dimension higher. This is because the grim-reaper is collapsed, but the bowl is noncollaped. Hence, one has to worry about the potential scenario that a blowup limit could be a wing-like translator with its sides modelled on the $2$-dimensional bowl times $\mathbb{R}$. This is illustrated in Figure \ref{figure_wing}. More generally, this lead to the question:

\begin{figure}
\begin{tikzpicture}[scale=0.8,x=1cm,y=1cm] \clip(-8,8) rectangle (8,-1);
\draw [samples=100,rotate around={45:(0,0)},xshift=0cm,yshift=0cm,domain=0.1:10)] plot (\x,1/\x);
\draw (2.2,4) [rotate around={-45:(3.7,4)}, partial ellipse=-90:0:1.5cm and 0.3cm];
\draw (2.2,4) [dashed, rotate around={-45:(3.7,4)}, partial ellipse=0:90:1.5cm and 0.3cm];
\draw (-2.2,4) [rotate around={45:(-3.7,4)}, partial ellipse=180:270:1.5cm and 0.3cm];
\draw (-2.2,4) [dashed, rotate around={45:(-3.7,4)}, partial ellipse=90:180:1.5cm and 0.3cm];
\draw (3,5) [rotate around={-45:(4.8,5)}, partial ellipse=-90:0:1.8cm and 0.3cm];
\draw (3,5) [dashed, rotate around={-45:(4.8,5)}, partial ellipse=0:90:1.8cm and 0.3cm];
\draw (-3,5) [rotate around={45:(-4.8,5)}, partial ellipse=180:270:1.8cm and 0.3cm];
\draw (-3,5) [dashed, rotate around={45:(-4.8,5)}, partial ellipse=90:180:1.8cm and 0.3cm];
\draw (1.4,3) [rotate around={-45:(2.6,3)}, partial ellipse=-90:0:1.2cm and 0.3cm];
\draw (1.4,3) [dashed, rotate around={-45:(2.6,3)}, partial ellipse=0:90:1.2cm and 0.3cm];
\draw (-1.4,3) [rotate around={45:(-2.6,3)}, partial ellipse=180:270:1.2cm and 0.3cm];
\draw (-1.4,3) [dashed, rotate around={45:(-2.6,3)}, partial ellipse=90:180:1.2cm and 0.3cm];
\draw (4,5.8) [blue, rotate around={45:(4,5.8)}, partial ellipse=180:360:0.27cm and 0.1cm];
\draw (4,5.8) [blue, dashed,  rotate around={45:(4,5.8)}, partial ellipse=0:180:0.27cm and 0.1cm];
\draw (-4,5.8) [blue, rotate around={-45:(-4,5.8)}, partial ellipse=180:360:0.27cm and 0.1cm];
\draw (-4,5.8) [blue, dashed, rotate around={-45:(-4,5.8)}, partial ellipse=0:180:0.27cm and 0.1cm];
\draw (3,4.7) [blue, rotate around={45:(3,4.7)}, partial ellipse=180:360:0.27cm and 0.1cm];
\draw (3,4.7) [blue, dashed, rotate around={45:(3,4.7)}, partial ellipse=0:180:0.27cm and 0.1cm];
\draw (-3,4.7) [blue, rotate around={-45:(-3,4.7)}, partial ellipse=180:360:0.27cm and 0.1cm];
\draw (-3,4.7) [blue, dashed, rotate around={-45:(-3,4.7)}, partial ellipse=0:180:0.27cm and 0.1cm];
\draw (2,3.6) [blue, rotate around={45:(2,3.6)}, partial ellipse=180:360:0.27cm and 0.1cm];
\draw (2,3.6) [blue, dashed, rotate around={45:(2,3.6)}, partial ellipse=0:180:0.27cm and 0.1cm];
\draw (-2,3.6) [blue, rotate around={-45:(-2,3.6)}, partial ellipse=180:360:0.27cm and 0.1cm];
\draw (-2,3.6) [blue, dashed, rotate around={-45:(-2,3.6)}, partial ellipse=0:180:0.27cm and 0.1cm];
\node at (0,5) {\Huge ?};
\draw (1,1)--(7,7) [red, dashed];
\draw (-1,1)--(-7,7) [red, dashed];
\node [red] at (3.5,2.5) {\small asymptotic};
\node [red] at (3.5,2.2) {\small lines};
\end{tikzpicture}
\caption{One of the most worrisome potential singularity models would be a three-dimensional noncollapsed wing-like flow in $\mathbb{R}^4$ with its sides modelled on the $2$-dimensional bowl times $\mathbb{R}$.}\label{figure_wing}
\end{figure}
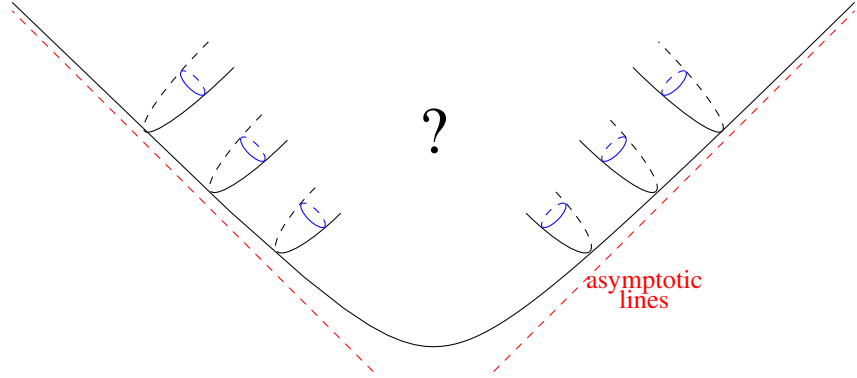

\begin{question}[noncollapsed wing-like translators]\label{question1}
Do there exist any noncollapsed wing-like translators in $\mathbb{R}^4$?
\end{question}

Here, with wing-like we mean any solution that is asymptotic to a wedge, i.e. a $2$-dimensional convex cone with angle strictly less than $\pi$.  Such solutions would cause a major complication for understanding the flow through singularities in $\mathbb{R}^4$, and would also be a major obstacle for the construction of a flow with surgery. More generally, one also has to worry about blowup limits that are not necessarily self-similar:

\begin{question}[noncollapsed wing-like ancient flows]\label{question2}
Do there exist any noncollapsed wing-like ancient flows in $\mathbb{R}^4$?
\end{question}

\bigskip

\subsection{Main results}

To address Question \ref{question1} and Question \ref{question2} (and even to properly define the notion of wing-like), we study the blowdown of the time slices. To this end, let $M_t=\partial K_t\subset\mathbb{R}^4$ be a noncollapsed ancient solution of the mean curvature flow. We recall that such flows are always convex thanks to \cite[Theorem 1.10]{HaslhoferKleiner_meanconvex}.

\begin{definition}[blowdown]\label{def_blowdown}
Given any time $t_0$, the \emph{blowdown} of $K_{t_0}$ is defined by
\begin{equation}
\check{K}_{t_0}:=\lim_{\lambda\to 0} \lambda\cdot K_{t_0}.
\end{equation}
\end{definition}

For example, if the flow is a three-dimensional bowl soliton, then its blowdown is a halfline. On the other hand, if the flow is a $2$-dimensional bowl soliton times $\mathbb{R}$, then its blowdown is a 2-dimensional halfplane. The scenario of a wing-like flow would correspond to the case where the blowdown is a wedge, i.e. a $2$-dimensional convex cone with angle strictly less than $\pi$.\\

Our main theorem provides a complete classification of all possible blowdowns of noncollapsed flows in $\mathbb{R}^4$:

\begin{theorem}[blowdown of noncollapsed flows in $\mathbb{R}^4$]\label{main_thm}
Let $M_t=\partial K_t$ be an ancient noncollapsed mean curvature flow in $\mathbb{R}^4$. Then its blowdown is time-independent and is
\begin{itemize}
\item either a point (which only happens if the solution is compact),
\item or a halfline,
\item or a line (which only happens for $S^2\times\mathbb{R}$ or oval times $\mathbb{R}$),
\item or a $2$-dimensional halfplane (which only happens if the solution is a $2$-dimensional bowl soliton times $\mathbb{R}$),
\item or a $2$-dimensional plane (which only happens for $S^1\times\mathbb{R}^2$),
\item or a $3$-dimensional hyperplane (which only happens for flat $\mathbb{R}^3$).
\end{itemize}
In particular, the blowdown can never be a wedge.
\end{theorem}

Theorem \ref{main_thm} shows that the asymptotic structure of noncollapsed flows, i.e. the flows that are actually relevant for singularity analysis, is much more rigid than the asymptotic structure of arbitrary convex flows. In fact, the example of the $\Delta$-wings illustrates that even for $2$-dimensional translators in $\mathbb{R}^3$ one can get the wedge of any angle less than $\pi$ as blowdown. In $\mathbb{R}^4$ there is a zoo of collapsed ancient convex solutions, see e.g. \cite{Wang_convex,HIMW}, whose blowdown can be much more arbitrary than the ones in Theorem \ref{main_thm}.\\

As an immediate corollary, we can answer Question \ref{question2}. Namely, since the blowdown can never be a wedge, we can rule out the  potential scenario of noncollapsed wing-like ancient flows in $\mathbb{R}^4$:

\begin{corollary}[Non-existence of wing-like noncollapsed flows]\label{cor1}
Wing-like noncollapsed ancient flows in $\mathbb{R}^4$ do not exist.
\end{corollary}

In particular, this also answers Question \ref{question1}:

\begin{corollary}[Non-existence of wing-like noncollapsed translators]\label{cor2}
Wing-like noncollapsed translators in $\mathbb{R}^4$ do not exist.
\end{corollary}

Corollary \ref{cor1} and Corollary \ref{cor2} rule out some of the most worrisome potential singularity models for the mean curvature flow of $3$-dimensional mean-convex hypersurfaces in $\mathbb{R}^4$ (and more generally also for mean-convex flows in $4$-manifolds, since after blowup the ambient manifold always becomes Euclidean).

\bigskip
\subsection{Outline and further results}

Let us now outline our approach. Let $M_t=\partial K_t$ be an ancient noncollapsed mean curvature flow in $\mathbb{R}^4$. We can assume that the solution is noncompact and nonflat, as otherwise the blowdown would simply be a point or a $3$-dimensional hyperplane, respectively.\\

As we explain in Section \ref{sec_coarse}, it follows from the general theory of noncollapsed flows \cite{HaslhoferKleiner_meanconvex} that the blowdown
\begin{equation}
\check{K}:=\lim_{\lambda\to 0} \lambda\cdot K_{t_0}.
\end{equation}
always is a convex cone of dimension at most $2$ (also, this limit is in fact independent of $t_0$). If $\check K$ splits off a line, then so does the flow $M_t=\partial K_t$, and we are done by the classification of $2$-dimensional noncollapsed flows in $\mathbb{R}^3$ by  Brendle-Choi \cite{BC} and Angenent-Daskalopoulos-Sesum \cite{ADS2}.  After these reductions, the task is thus to rule out the scenario where $\check K$ is wedge.\\
We next consider the tangent flow at $-\infty$. Flows with a \emph{neck-tangent flow at $-\infty$}, i.e. with
\begin{equation}\label{neck_tangent}
\lim_{\lambda \rightarrow 0} \lambda M_{\lambda^{-2}t}=\mathbb{R}\times S^2(\sqrt{-4t}),
\end{equation}
have already been classified by \cite{BC2}. We can thus assume that we have a \emph{bubble-sheet tangent flow at $-\infty$}, i.e. that
\begin{equation}\label{bubble_tangent_intro}
\lim_{\lambda \rightarrow 0} \lambda M_{\lambda^{-2}t}=\mathbb{R}^2\times S^1(\sqrt{-2t}).
\end{equation}

In Section \ref{sec_fol_barriers}, to facilitate various calibration and barrier arguments, we build a foliation for flows close to a bubble sheet in $\mathbb{R}^4$. We do this by shifting and rotating the $2$-dimensional shrinker foliation in $\mathbb{R}^3$ from \cite{ADS}.\\

In Section \ref{Sec_set_up}, we set up the fine bubble-sheet analysis, which generalizes the fine neck-analysis that played a major role in \cite{ADS,ADS2,BC,BC2,CHH,CHHW}. Given any space-time point $X_0$, we consider the renormalized flow
\begin{equation}
\bar M_\tau = e^{\frac{\tau}{2}} \, \left( M_{t_0-e^{-\tau}} - x_0\right).
\end{equation}
Then, the hypersurfaces $\bar M_\tau$ converge for $\tau\to -\infty$ to the cylinder
\begin{equation}
\Gamma=\mathbb{R}^2\times S^1(\sqrt{2}).
\end{equation}
In particular, $\bar M_\tau$ can be written as the graph of a function $u=u_{X_0}(\cdot,\tau)$ with small norm over $\Gamma\cap B_{2\rho(\tau)}$, where $\rho(\tau)\to \infty$ as $\tau\to -\infty$. The goal is then to derive a sharp asymptotic expansion for the graph function $u$.\\
The analysis over $\Gamma$ is governed by the Ornstein-Uhlenbeck operator
\begin{equation}
\mathcal L=\frac{\partial^2}{\partial x_1^2} +\frac{\partial^2}{\partial x_2^2} + \frac{1}{2} \, \frac{\partial^2}{\partial \theta^2}  - \frac{1}{2} \, x_1 \, \frac{\partial}{\partial x_1} - \frac{1}{2} \, x_2 \, \frac{\partial}{\partial x_2} + 1.\end{equation}
The operator $\mathcal{L}$ has $5$ unstable eigenfunctions, namely
\begin{equation}\label{eigenfunctions_unstable}
1,\cos\theta,\sin\theta,x_1,x_2,
\end{equation}
and $7$ neutral eigenfunctions, namely
\begin{equation}\label{expansion_neutral}
x_1\cos\theta,x_1\sin\theta,x_2\cos\theta,x_2\sin\theta,x_1^2-2,x_2^2-2,x_1x_2.
\end{equation}
Using the ODE-lemma from Merle-Zaag \cite{MZ}, we show that for $\tau\to -\infty$ either the unstable mode is dominant or the neutral mode is dominant. Furthermore, considering instead an expansion for
\begin{equation}
\tilde{M}_\tau=S(\tau) \bar{M}_\tau,
\end{equation}
where the fine-tuning rotation $S(\tau)\in \mathrm{SO}(4)$ is obtained via the implicit function theorem, we can assume that in the neutral mode case the truncated function $\hat u$ is orthogonal to the $\theta$-dependent functions from \eqref{expansion_neutral}.\\

In Section \ref{sec_neutral_mode}, we consider the case where the neutral mode is dominant. By the reductions from above we can assume that the blowdown $\check{K}$ is a convex cone that does not contain any line. Furthermore, rotating coordinates, we can arrange that  $\check{K}$ contains the positive $x_1$-axis, namely
\begin{equation}
\{ \lambda e_1\, |\, \lambda \geq 0\} \subseteq \check{K}.
\end{equation}
In this setting, we have:
\begin{theorem}[blowdown in neutral mode]\label{thm_blowdown_neutral_intro}
If the neutral mode is dominant, then its blowdown is a halfline, namely
\begin{equation}
\check{K}=\{ \lambda e_1\, |\, \lambda \geq 0\} .
\end{equation}
Moreover,
\begin{equation}
\lim_{\tau \to -\infty}\frac{\hat u(\cdot,\tau)}{\|\hat u(\cdot,\tau)\|}= -c (x_2^2-2),
\end{equation}
where $c>0$.
\end{theorem}
To prove this, remembering \eqref{expansion_neutral} and the fine-tuning, we first show that along a subsequence we have
\begin{equation}
\lim_{\tau_{i_m} \to -\infty}\frac{\hat u(\cdot,\tau_{i_m})}{\|\hat u(\cdot,\tau_{i_m})\|}= q_{11}(x_1^2-2)+ q_{22}(x_2^2-2)+2q_{12}x_1x_2.
\end{equation}
Using Brunn's concavity principle, we show that $Q=\{q_{ij}\}$ is a nontrivial semi-negative definite $2\times2$-matrix.\\
The crucial step is then to relate the algebra of the quadratic form $Q$ with the geometry of the blowdown $\check K$. Specifically, we show that
\begin{equation}\label{intro_key_inclusion}
\check K \subseteq \ker Q.
\end{equation}
The basic idea for this is that in directions $v\notin \ker Q$ we have an inwards quadratic bending, which implies that $v$ is a ``short" direction, i.e. $v\notin \check K$. This can be made precise using again Brunn's concavity principle. Once \eqref{intro_key_inclusion} is established, Theorem \ref{thm_blowdown_neutral_intro} follows easily. In fact, our argument also gives a quantitative estimate for the diameter of the short directions:

\begin{corollary}[diameter of level sets]\label{diameter bound rough_intro}
If $M_t=\partial K_t$ is as above, then for any $\delta>0$ we have
\begin{equation*}
\bar M_\tau\cap \{x_1=0\} \subseteq B_{e^{\delta|\tau|}}(0)
\end{equation*}
for $\tau \ll 0$.
\end{corollary}

In Section \ref{sec_unstable}, we consider the case where the unstable mode is dominant. Our main technical result for such flows is the following:

\begin{theorem}[Fine bubble-sheet theorem] \label{thm_fine_bubble_sheet_intro} Let $\{M_t\}$ be an ancient noncollapsed flow in $\mathbb{R}^4$, with a bubble sheet tangent flow at $-\infty$, whose unstable mode is dominant. Then
there exist universal constants $ a_1,a_2$, such that for every $X$, after suitable re-centering in the $x_3x_4$-plane, the truncated graph function $\check u^X(\cdot,\tau)$ of the renormalized flow $\bar{M}_\tau^X$ satisfies
\begin{align}
 \check u^X=e^{\tfrac{\tau}{2}}\left(a_1 x_1 +a_2x_2 \right)+ o(e^{\tfrac{\tau}{2}})
\end{align}
for $\tau\ll 0$ depending only on an upper bound on the bubble-sheet scale $Z(X)$. Moreover, the expansion parameters satisfy
\begin{equation}
|a_1|+|a_2| >0.
\end{equation}
\end{theorem}

Theorem \ref{thm_fine_bubble_sheet_intro} shows that the bubble-sheet increases its bubble-size slightly, in a precisely described way, as one moves in direction of the vector $(a_1,a_2)$. The constants $ a_1,a_2$ are genuine constants of the flow. For example, if $M_t$ is $\mathbb{R}$ times a $2$-dimensional bowl translating in $x_2$-direction, then $a_1=0$ and $a_2$ is proportional to the reciprocal of the speed of translation.\\

The fine-bubble sheet theorem (Theorem \ref{thm_fine_bubble_sheet_intro}) generalizes the fine-neck theorem from our prior work \cite{CHH,CHHW}. To prove it, we follow the scheme from our prior work with the necessary modifications. In particular, we now expand in terms of the unstable eigenfunctions from \eqref{eigenfunctions_unstable}, and we use the new foliation from Section \ref{sec_fol_barriers}. Another new step is to show that the unstable mode is dominant even after removing the fine-tuning rotation.\\

In Section \ref{sec_unstable_end}, we play the following end game. Let $M_t=\partial K_t$ be a noncompact ancient noncollapsed flow in $\mathbb{R}^4$, with bubble-sheet tangent flow at $-\infty$, whose unstable mode is dominant, and suppose towards a contradiction that its blowdown is a wedge. Then taking suitable limits along the sides, we see $\mathbb{R}$ times a $2$-dimensional bowl translating soliton. As observed above, the vector $(a_1,a_2)$ points in the translation direction. However, since the limits obtained along the two different sides of the wedge translate in different directions this gives the desired contradiction. In fact, the argument also works for a degenerate wedge (i.e. a halfline) and thus proves:

\begin{theorem}[classification in unstable mode]\label{ruling_out_unstable_intro}
The only noncompact ancient noncollapsed flow in $\mathbb{R}^4$, with bubble-sheet tangent flow at $-\infty$, whose unstable mode is dominant, is $\mathbb{R}\times2d$-bowl.
\end{theorem}

Theorem \ref{ruling_out_unstable_intro}, in addition to ruling out the wedge blowdown, in fact completes the classification of noncompact ancient noncollapsed flows in $\mathbb{R}^4$ in case the unstable mode is dominant. We will address the neutral mode case in subsequent work, based on Theorem \ref{thm_blowdown_neutral_intro} and Corollary \ref{diameter bound rough_intro}.

\bigskip

\noindent\textbf{Acknowledgments.}
KC has been supported by KIAS Individual Grant MG078901 and a POSCO Science Fellowship. RH has been supported by an NSERC Discovery Grant and a Sloan Research Fellowship. OH has been supported by the Koret Foundation early career award and ISF grant 437/20. We are very grateful to Wenkui Du and the anonymous referees for very detailed and helpful comments.\\

\bigskip

\section{Coarse properties of ancient noncollapsed flows}\label{sec_coarse}

Recall that by \cite[Thm. 1.14]{HaslhoferKleiner_meanconvex} and \cite{CM_uniqueness}, if $M_t=\partial K_t$ is an ancient noncollapsed\footnote{By \cite{Brendle_inscribed,HaslhoferKleiner_inscribed} we can always take $\alpha=1$ as noncollapsing constant.} mean curvature flow in $\mathbb{R}^4$, that is noncompact and nonflat, then, up to rotation, either 
\begin{equation}\label{neck_tangent}
\lim_{\lambda \rightarrow 0} \lambda K_{\lambda^{-2}t}=\mathbb{R}\times D^3(\sqrt{-4t}),
\end{equation}
or
\begin{equation}\label{bubble_tangent}
\lim_{\lambda \rightarrow 0} \lambda K_{\lambda^{-2}t}=\mathbb{R}^2\times D^2(\sqrt{-2t}).
\end{equation}
In the case \eqref{neck_tangent} we say that $M_t$ has a \emph{neck tangent flow at $-\infty$}, whereas in case \eqref{bubble_tangent} we say that $M_t$ has a \emph{bubble-sheet tangent flow at $-\infty$}.\\

In the neck tangent case, it follows from Brendle-Choi \cite{BC2} (or alternatively from \cite{CHHW}) that $M_t$ is either a round shrinking $S^2\times\mathbb{R}$ or a $3$-dimensional bowl. Hence, we can focus on the bubble-sheet case. \\

\subsection{The blowdown of time slices}
In this section, we establish several elementary properties of the blowdown of time slices (Definition \ref{def_blowdown}).

\begin{proposition}[basic properties of blowdown]\label{2dcone_blow_up}
For any ancient noncollapsed mean curvature flow $M_t=\partial K_t$ in $\mathbb{R}^4$, that is noncompact and nonflat, the blowdown $\check{K}_{t_0}$ is a convex cone of dimension at most $2$.
\end{proposition}

\begin{proof}
Assume $t_0=0$ and $0\in M_0$ for ease of notation. By convexity (see  \cite[Theorem 1.10]{HaslhoferKleiner_meanconvex}), we have $ \lambda\cdot K_0\subseteq \mu\cdot K_0$ whenever $\lambda\leq \mu$. Hence,
\begin{equation}
\check{K}_0:=\lim_{\lambda\to 0} \lambda\cdot {K_{0}}=\bigcap_{\lambda>0} \lambda\cdot {K_{0}}
\end{equation}
exists and is convex. Since $K_0$ is noncompact, $\check{K}_0$ contains a halfline.
Clearly, we have $\lambda\cdot \check{K}_0=\check{K}_0$ for any $\lambda>0$, i.e. $\check{K}_0$ is a cone. \\
As the flow is moving inwards, for every $t<0$ we have
\begin{equation}
\lambda K_0 \subseteq \lambda K_{\lambda^{-2}t}.  
\end{equation}
Thus, 
\begin{equation}
\check{K}_0 = \lim_{\lambda\rightarrow 0}\lambda K_0 \subseteq  \lim_{\lambda \rightarrow 0}\lambda K_{\lambda^{-2}t},
\end{equation}
where the limit in the right hand side is described (after suitable rotation) either by \eqref{neck_tangent} or by \eqref{bubble_tangent}. In either case, we infer that
\begin{equation}
\check{K}_0 \subseteq \mathbb{R}^2\times \{0\},
\end{equation}
so $\check{K}_0$ is at most $2$-dimensional. This proves the proposition.
\end{proof}

\begin{proposition}[dimension reduction]
Let $M_t=\partial K_t$ be an ancient noncollapsed mean curvature flow in $\mathbb{R}^4$. If $\check{K}_{t_0}$ contains a line, then $M_t$ is either a static hyperplane, a round shrinking $\mathbb{R}\times S^2$, a round shrinking $\mathbb{R}^2\times S^1$, a $2$-dimensional bowl times $\mathbb{R}$, or a $2$-dimensional ancient oval times $\mathbb{R}$.
\end{proposition}

\begin{proof}
If $\check{K}_{t_0}$ contains a line, then so does $K_{t_0}$. Hence, the flow $M_t=\partial K_t$ splits off an $\mathbb{R}$-factor. By the classification of $2$-dimensional noncollapsed flows in $\mathbb{R}^3$ by Angenent-Daskalopoulos-Sesum \cite{ADS} and Brendle-Choi \cite{BC}, this implies the assertion.
\end{proof}

We will now show that the blowdown $\check{K}_{t_0}$ is independent of $t_0<T_{\mathrm{ext}}(\mathcal M)$, where $T_\mathrm{ext}(\mathcal M)$ denotes the extinction time of our flow. To show this, we need the following lemma (we write $\nu$ for the outward unit normal).

\begin{lemma}[interior balls]\label{some_distance}
For every $\theta_0>0$ and $H_0<\infty$ there exists some $\delta=\delta(H_0,\theta_0)>0$ with the following significance: If $p\in M_{t_0}$ is such that $H(p)\leq H_0$ and $\langle -\nu(p),v\rangle\geq \theta_0$ for every $v\in  \check{K}_{t_0}$ with $||v||=1$, then we have $B(p+v,\delta) \subseteq K_{t_0}$ for every $v\in \check{K}_{t_0}$ with $||v||\geq 1$.
\end{lemma}

\begin{proof}
First, note that as $\lim_{\lambda\rightarrow 0}\lambda K_{t_0}=\lim_{\lambda\rightarrow 0}\lambda (K_{t_0}-p)$,  we have $p+sv\in K_{t_0}$ for every $v\in\check{K}_{t_0}$ with $||v||=1$ and every $s\geq 0$.  The convexity of $K_{t_0}$ implies that the function $(0,\infty)\ni s\mapsto \mathrm{dist}(p+sv, M_{t_0})$ is concave and positive on $K_{t_0}$, and thus, nondecreasing. Taking an interior tangent ball at $p$ of radius $\alpha/H_0$ yields the desired result.     
\end{proof}

\begin{proposition}[time-independence of blowdown]\label{height vector}
The blowdown $\check{K}_{t_0}$ is independent of $t_0<T_{\mathrm{ext}}(\mathcal{M})$.
\end{proposition}
\begin{proof}
It suffices to prove this in case $M_t=\partial K_t$ in $\mathbb{R}^4$ is noncompact and strictly convex.
Fix some $t_0$ and consider the set  $I$ of all $t\in [t_0,\mathrm{T}_{\mathrm{ext}}(\mathcal{M}))$ such that $\check{K}_{t_0}=\check{K}_{t}$. Since $K_{t_2} \subseteq K_{t_1}$ for $t_2 \geq t_1$ we clearly have that $\check{K}_{t_2} \subseteq \check{K}_{t_1}$  for $t_2 \geq t_1$, so $I$ is a (potentially degenerate) subinterval of $[t_0,\mathrm{T}_{\mathrm{ext}}(\mathcal{M}))$ that contains $t_0$.  Letting $t_1\in I$, and taking any $p\in \partial K_{t_1}$, strict convexity implies that  
\begin{equation}
\inf_{v\in \check{K}_{t_1}-\{0\}} \left\langle - \nu(p),\frac{v}{||v||}\right\rangle >0.
\end{equation}
By the lemma above, there exists some $\delta>0$ such that for every $||v||\geq 1$, $B(p+v,\delta) \subseteq K_{t_1}$. Therefore, by avoidance, $p+v\in K_t$ for all $t$ with $|t-t_1|\leq \delta^2/6$, so all such $t$ are also in $I$.  \\
Thus, if $I\neq [t_0,T_{\mathrm{ext}}(\mathcal{M}))$ then $I= [t_0,t_1)$ for some $t_1<T_{\mathrm{ext}}(\mathcal{M})$. We want to show that this is impossible. Suppose, for the sake of contradiction, that there exists some $v\in \check{K}_{t_0} \setminus  \check{K}_{t_1}$. Given any $p\in M_{t_1}$, by smoothness, we can find $t\in I$ arbitrarily close to $t_1$ and $p_t\in M_t$ such $p_t\rightarrow p$ and such that $H(p_t)\leq H(p)+1$. As above, Lemma \ref{some_distance} and avoidance imply that   
\begin{equation}
\langle \nu(p_t),v\rangle\rightarrow 0,
\end{equation}
so $v\in T_pM_{t_1}$. But since $p\in M_{t_1}$ was arbitrary, this implies that $M_{t_1}$ splits a line in the direction $v$, contradicting the strict convexity of $M_{t_1}$. 
\end{proof}

We end this section by giving a (well-known) description of the blowdown in terms of graphical directions:

\begin{proposition}[alternative description of blowdown]\label{blow_down_normals}
Let $K_t=\partial M_t$ be a noncollapsed strictly convex flow. Then 
\begin{equation}
\check{K}_{t_0}=\{\omega\in \mathbb{R}^{n+1}\;\mathrm{s.t }\; \langle \nu(p), \omega\rangle <0\;\mathrm{for\; every }\;p\in M_{t_0}\}\cup \{0\}.
\end{equation}
\end{proposition}

\begin{proof} Let $\Omega=\{\omega\in \mathbb{R}^{n+1}\;\textrm{s.t } \langle \nu(p), \omega\rangle <0\;\textrm{for every }p\in M_{t_0}\}$. The containment $\check{K}_{t_0}\subseteq \Omega \cup \{0\}$ is clear, as for every $p\in M_{t_0}$ and $v\in\check{K}_{t_0}$ we have $p+sv\in K_{t_0}$ for every positive $s$.  Conversely, if $v\in \Omega$,  for every $p\in M_{t_0}$ we have that $p+sv \in K_{t_0}$ for small values of $s$. Thus, if $p+s_0v\in M_{t_0}$ for some $s_0>0$  then for $s>s_0$ with $s-s_0$ sufficiently small $p+sv\in K_{t_0}$, which contradicts the concavity of the distance to the boundary. Thus, the entire line $\{p+sv\;|\;s\in [0,\infty)\}$ is contained in $K_{t_0}$, so $v\in \check{K}_{t_0}$. 
\end{proof}

\bigskip

\subsection{The bubble-sheet scale} 

Let $\mathcal M$ be an ancient mean curvature flow, with a bubble-sheet tangent flow at $-\infty$. Given a point $X=(x,t)\in \mathcal{M}$ and a scale $r>0$, we consider the flow
\begin{equation}
\mathcal{M}_{X,r}=\mathcal{D}_{1/r}(\mathcal M -X),
\end{equation}
which is obtained from $\mathcal M$ by translating $X$ to the space-time origin and parabolically rescaling by $1/r$. Here, $\mathcal{D}_\lambda (x,t) = (\lambda x,\lambda^2 t)$.

\begin{definition}[$\eps$-cylindrical]\label{eps_cyl}
We say that $\mathcal M$ is \emph{$\varepsilon$-cylindrical around $X$ at scale $r$}, if $\mathcal{M}_{X,r}$ is $\varepsilon$-close in $C^{\lfloor1/\varepsilon \rfloor}$ in $B(0,1/\varepsilon)\times [-2,-1]$ to the evolution of a round shrinking bubble-sheet cylinder with radius $r(t)=\sqrt{-2t}$ and axis through the origin.
\end{definition}

Note that in \cite{CHH} and \cite{CHHW}, being $\eps$-cylindrical is defined by closeness to the evolution of a neck, rather than the evolution of a bubble-sheet. We hope that the benefits of the analogies coming from overloading the term outweigh the potential confusion. In any case, throughout the present paper $\varepsilon$-cylindrical is always meant in the sense of Definition \ref{eps_cyl}.\\    

Given any point $X=(x,t)\in\mathcal{M}$, we analyze the solution around $X$ at the diadic scales $r_j=2^j$, where $j\in \mathbb{Z}$. 

\begin{theorem}\label{thm_finding_necks}
For any small enough $\varepsilon>0$, there is a positive integer $N=N(\varepsilon)<\infty$ with the following significance. If $\mathcal M$ is an ancient flow asymptotic to a bubble-sheet, which is not the round shrinking bubble-sheet cylinder, then for every $X\in \mathcal{M}$ there exists an integer $J(X)\in\mathbb{Z}$ such that
\begin{equation}\label{eq_thm_quant1}
\textrm{$\mathcal M$ is not $\varepsilon$-cylindrical around $X$ at scale $r_j$ for all $j<J(X)$},
\end{equation}
and
\begin{equation}\label{eq_thm_quant2}
\textrm{$\mathcal M$ is $\tfrac{\varepsilon}{2}$-cylindrical around $X$ at scale $r_j$ for all $j\geq J(X)+N$}.
\end{equation}
\end{theorem}
\begin{proof}
The proof, based on quantitative differentiation (c.f. \cite{CHN_stratification}), is identical to the one of \cite[Theorem 2.7]{CHH}.
\end{proof}

We fix a small enough parameter $\varepsilon>0$ quantifying the quality of the bubble-sheet for the rest of the paper.

\begin{definition}[bubble-sheet scale]
The \emph{bubble-sheet scale} of $X\in\mathcal{M}$ is defined by
\begin{equation}
Z(X)=2^{J(X)}.
\end{equation}
\end{definition}

\bigskip

\section{Foliations and barriers}\label{sec_fol_barriers}

The proof of Theorem \ref{main_thm} relies on the spectral analysis of the equation describing the evolution by mean curvature flow of a graph over the cylinder. Importantly, the flow is not an entire graph over a cylinder, but rather, a graph over a large portion of it. Thus, one needs to introduce some truncation, and control the error introduced by it. Moreover, one needs quantitative bounds on the graphical radius - the radius in which the flow is a small graph over the cylinder. In the neck-singularity setting, both of these goals are achieved in \cite{ADS} by using a foliation whose leaves are fixed points of the rescaled MCF.  In this section, we construct analogues of that foliation which are suited for the analysis over the bubble-sheet cylinder $\mathbb{R}^2\times S^1$. \\         

We recall from Angenent-Daskalopoulos-Sesum that there exists some $L_0>1$ such that  for every $a\geq L_0$ and $b>0$ there are self-shrinkers 
\begin{align}
{\Sigma}_a &= \{ \textrm{surface of revolution with profile } r=u_a(x_1), 0\leq x_1 \leq a\},\\
\tilde{{\Sigma}}_b &= \{ \textrm{surface of revolution with profile } r=\tilde{u}_b(x_1), 0\leq x_1 <\infty\},\nonumber
\end{align}
as illustrated in \cite[Fig. 1]{ADS}, see also \cite{KM}.     We will refer to these shrinkers as ADS-shrinkers and KM-shrinkers, respectively. Here, the parameter $a$ captures where the  concave functions $u_a$ meet the $x_1$-axis, namely $u_a(a)=0$, and the parameter $b$ is the asymptotic slope of the convex functions $\tilde{u}_b$, namely $\lim_{x_1\to \infty} \tilde{u}_b'(x_1)=b$. A detailed description of these shrinkers can be found in \cite[Sec. 8]{ADS}. We recall:

\begin{lemma}[Lemmas 4.9 and 4.10 in \cite{ADS}] \label{lower_d_ADS}
There exists some $b_0>0$ such that the self-shrinkers ${\Sigma}_a,\tilde{{\Sigma}}_b$ and the cylinder ${\Sigma}:=\{x_2^2+x_3^2=2\}\subset \mathbb{R}^3$ foliate the region
\begin{equation}
\{(x_1,x_2,x_3)\in \mathbb{R}^3\;|\; x_2^2+x_3^2 \leq b_0^2x_1^2\;\textrm{and } x_1\geq L_0\}.
\end{equation}
Moreover, denoting by $\nu_{\textrm{fol}}$ the outward unit normal of this family, we have that
\begin{equation}\label{eq_calibration}
\mathrm{div}(e^{-x^2/4} \nu_{\mathrm{fol}})=0,
\end{equation}
i.e. the shrinker family forms a calibration for the Gaussian area.
\end{lemma}

We will now shift and rotate this foliation to construct a suitable foliation in $\mathbb{R}^4$:

\begin{definition}[bubble-sheet foliation]\label{new_fol_def}
For every $a\geq L_0$, we denote by $\Gamma_a$ the doubly-rotationally symmetric hypersurface in $\mathbb{R}^4$ given by
\begin{equation}
\Gamma_a=\{(r\cos\theta ,r\sin\theta,x_3,x_4)\in  \mathbb{R}^4:\theta\in [0,2\pi), (r-1,x_3,x_4) \in {\Sigma}_a \}.
\end{equation}
Similarly, for each $b>0$ we denote by ${\tilde{\Gamma}}_b$ the doubly-rotationally symmetric hypersurface in $\mathbb{R}^4$ given by
\begin{equation}
{\tilde{\Gamma}}_b=\{(r\cos\theta ,r\sin\theta,x_3,x_4)\in  \mathbb{R}^4:\theta\in [0,2\pi), (r-1,x_3,x_4) \in \tilde{{\Sigma}}_b \}.
\end{equation}
\end{definition}

\begin{lemma}[Foliation lemma]\label{foli_lemma}
There exist $b_0>0$ and $L_1\geq 3$ such that the hypersurfaces ${\Gamma}_a$, ${\tilde{\Gamma}}_b$, and the cylinder ${\Gamma}:=\mathbb{R}^2\times S^1(\sqrt{2})$ foliate the domain 
\begin{equation}
{\Omega}:=\left\{(x_1,x_2,x_3,x_4)\in \mathbb{R}^4\;|\; x_3^2+x_4^2 \leq b_0^2(x_1^2+x_2^2)\;\textrm{and } x_1^2+x_2^2\geq L_1^2\right\}.
\end{equation}
Moreover, denoting by $\nu_{\mathrm{fol}}$ the outward pointing unit normal of this foliation, we have that 
\begin{equation}\label{negative_div}
\mathrm{div}(\nu_{\mathrm{fol}}e^{-|x|^2/4})\leq 0\;\;\;\textrm{inside the cylinder},
\end{equation} 
and
\begin{equation}\label{positive_div}
\mathrm{div}(\nu_{\mathrm{fol}}e^{-|x|^2/4})\geq 0\;\;\;\textrm{outside the cylinder}.
\end{equation}  
\end{lemma}

\begin{proof}
Let $b_0$ be as in Lemma \ref{lower_d_ADS}, and set $L_1=L_0+1$. The fact that $\Gamma_a$, ${\tilde{\Gamma}}_b$ and ${\Gamma}$ foliate ${\Omega}$ follows from Lemma \ref{lower_d_ADS} and Definition \ref{new_fol_def}.

Now, observe that for every element ${\Gamma}_{*}$ in the foliation of ${\Omega}$ we have
\begin{equation}
\mathrm{div}({\nu}_{\mathrm{fol}}e^{-|x|^2/4})=\big(H_{{\Gamma}_*}-\tfrac{1}{2}\langle x,\nu_{\mathrm{fol}} \rangle\big)e^{-|x|^2/4}.
\end{equation}
Hence, \eqref{negative_div} is equivalent to the condition
\begin{equation}\label{to_show_negative_div}
H_{{\Gamma}_*} \leq \tfrac{1}{2}\langle x,\nu_{\mathrm{fol}} \rangle.
\end{equation}
To show \eqref{to_show_negative_div}, note that by symmetry, it suffices to compute the curvatures $H_{{\Gamma}_{*}}$ of ${\Gamma}_{*}$ in the region $\{x_2=0,\;x_1>0\}$, where we can identify points and unit normals in ${\Gamma}_{*}$ with the corresponding ones in ${\Sigma}_*$, by disregarding the $x_2$-component. 
The relation between the mean curvature  of a surface ${\Sigma}_{*}\subset\mathbb{R}^3$ and its (unshifted) rotation ${\Gamma}_{*}\subset\mathbb{R}^4$ on points with $x_2=0$ and $x_1>0$ is given by 
\begin{equation}\label{rotate_H}
H_{{\Gamma}_{*}}=H_{{\Sigma}_{*}}+\frac{1}{x_1}\langle e_1,\nu \rangle,
\end{equation}  
where $e_1=(1,0,0)$.

For the ADS-shrinkers, the concavity of $u_a$ implies $\langle e_1,\nu_{\mathrm{fol}} \rangle\geq 0$, so using \eqref{rotate_H} and the shrinker equation we infer that
\begin{equation}\label{rot_cal}
H_{{{\Gamma}}_a}
= \frac{1}{2}\langle x-e_{1},\nu_{\mathrm{fol}} \rangle+\frac{1}{x_1-1} \langle e_1,\nu_{\mathrm{fol}} \rangle \leq \frac{1}{2}\langle x, \nu_{\mathrm{fol}}\rangle,  
\end{equation}
as in $\Omega\cap \{x_2=0,\;x_1>0\}$,  we have $x_1 \geq L_1 \geq 3$.

For the KM-shrinkers, the convexity of $\tilde{u}_b$ implies that $\langle e_1,\nu_{\mathrm{fol}} \rangle\leq  0$, so a similar calculation as in \eqref{rot_cal} gives
\begin{equation}
H_{{\tilde{\Gamma}}_b} \geq \frac{1}{2}\langle x, \nu_{\mathrm{fol}}\rangle.
\end{equation} 
This proves the lemma.
\end{proof}

\begin{corollary}[Inner Barriers]
The hypersurfaces ${\Gamma}_a$ are inner barriers for the renormalized mean curvature flow, in the following sense: Assume $\{K_{\tau}\}_{\tau\in [\tau_1,\tau_2]}$ are compact domains, the boundary of which evolve by renormalized mean curvature flow. Assume further that $K_{\tau}$  for every $\tau\in [\tau_1,\tau_2]$ contains the region bounded by ${\Gamma}_a$ and  $x_1^2+x_2^2=L_1^2$,  and   that $\partial K_{\tau}\cap {\Gamma}_a=\emptyset$ for all $\tau<\tau_2$. Then 
\begin{equation}
\partial K_{\tau_2}\cap {\Gamma}_a\subseteq \partial {\Gamma}_a.
\end{equation} 
\end{corollary}\label{lemma_inner_barrier}

\begin{proof}
Lemma \ref{foli_lemma} implies that the vector $\vec{H}+\frac{x^{\perp}}{2}$
points outwards of ${\Gamma}_a$. The result now follows from the maximum principle.
\end{proof}

\begin{remark}[Outer Barriers]
Although this is not needed in the convex case, the hypersurfaces $\tilde{{\Gamma}}_b$ are evidently outer barriers for the renormalized mean curvature flow.  
\end{remark}

\bigskip

\section{Setting up the fine bubble-sheet analysis}\label{Sec_set_up}

Throughout this section, let $\mathcal{M}$ be a noncollapsed ancient flow in $\mathbb{R}^4$ whose tangent flow for $t\to -\infty$ is a bubble-sheet. Assume further that $\mathcal{M}$ is not self-similarly shrinking. Given any space-time point $X_0$, we consider the renormalized flow
\begin{equation}
\bar M^{X_0}_\tau = e^{\frac{\tau}{2}} \, \left( M_{t_0-e^{-\tau}} - x_0\right).
\end{equation}
Then, $\bar M^{X_0}_\tau$ converges to the cylinder
\begin{equation}
\Gamma=\mathbb{R}^2\times S^1(\sqrt{2})
\end{equation}
as $\tau\rightarrow -\infty$. \\

Since the renormalized MCF is invariant under rotations,  the corresponding rotation vector fields appear as Jacobi fields in its  linearization. Therefore, to obtain a useful geometric information from the spectral analysis in the case where the neutral mode is dominant, one needs to make sure that those rotation vector fields are not the dominant ones. There are two ways that have been successfully employed in doing that: 
\begin{enumerate}
\item using a neck-improvement theorem to show that the rotations are nondominant (c.f. \cite[Sec. 4.3]{CHH}), or
\item rotating the evolution in such a way that the graphs are orthogonal to all rotations (c.f. \cite[Sec. 2]{BC}).
\end{enumerate}
The difference between the two methods, in terms of the argument, is where those rotations are dealt with: in the latter approach, the labor lies in showing that the modes of the linearization dominate the evolution even after one rotates the surfaces. In the former approach, the analysis in the neutral mode case is harder.

In  our current  setting, we have found it easier to use the second alternative, as in \cite{BC}.   \\

After normalizing we can assume that $Z(X_0)\leq 1$, and we can find a universal function $\rho(\tau)>0$ with
\begin{equation}\label{univ_fns}
\lim_{\tau \to -\infty} \rho(\tau)=\infty, \quad \textrm{and} -\rho(\tau) \leq \rho'(\tau) \leq 0,
\end{equation}
such that for every $S\in \mathrm{SO}(4)$ with $\sphericalangle(S(\Gamma),\Gamma)<\frac{1}{100} \rho(\tau)^{-3}$ the rotated surface $S(\bar{M}^{X_0}_\tau)$ is the graph of a function $u=u_S(\cdot,\tau)$ over $\Gamma \cap B_{2\rho(\tau)}(0)$, namely 
\begin{equation}
\{x + u(x,\tau) \nu_\Gamma(x): x \in \Gamma \cap B_{2\rho(\tau)}(0)\} \subset S(\bar{M}^{{X_0}}_\tau),
\end{equation} 
where $\nu_\Gamma$ denotes the outward pointing unit normal to $\Gamma$, such that
\begin{equation}\label{small_graph}
\|u(\cdot,\tau)\|_{C^4(\Gamma \cap B_{2\rho(\tau)}(0))} \leq \rho(\tau)^{-2}.
\end{equation}
For later use, let us also arrange that in the special case where the rotation matrix is the identity we have the better decay
\begin{equation}\label{smaller_graph}
\|u_{I}(\cdot,\tau)\|_{C^4(\Gamma \cap B_{2\rho(\tau)}(0))} \leq \rho(\tau)^{-4}.
\end{equation}
\bigskip

We fix a nonnegative smooth cutoff function $\chi$ satisfying $\chi(s)=1$ for $|s| \leq \frac{1}{2}$ and $\chi(s)=0$ for $|s| \geq 1$. We consider the truncated function
\begin{equation}
\hat{u}(x,\tau):=u(x,\tau) \chi\left(\frac{r}{\rho(\tau)}\right),
\end{equation}
where 
\begin{equation}\label{r_def}
r(x):=\sqrt{x_1^2+x_2^2}.
\end{equation}

\begin{proposition}[Orthogonality]\label{orthogonality}
There exists a differentiable function $S(\tau):=S^{X_0}(\tau)\in \mathrm{SO}(4)$, defined for $\tau$ sufficiently negative, such that, setting $u:=u_{S^{X_0}(\tau)}$,  we have
\begin{equation}
\int_{\Gamma\cap B_{\rho(\tau)}(0)} \langle Ax,\nu_\Gamma\rangle \hat{u}(x,\tau) e^{-\frac{|x|^2}{4}}=0,
\end{equation}
for every $A\in \mathrm{o}(4)$. Moreover, the matrix $A(\tau)=S'(\tau)S(\tau)^{-1}\in \mathrm{o}(4)$ satisfies $A_{12}(\tau)=0$ and $A_{34}(\tau)=0$ for all $\tau\ll 0$. 
\end{proposition}

\begin{proof}
Let $\textrm{Gr}(2,\mathbb{R}^4)$ be the space of $2$-dimensional planes through the origin in $\mathbb{R}^4$. The rotation group $\mathrm{O}(4)$ acts transitively on $\textrm{Gr}(2,\mathbb{R}^4)$ with stabilizer $\mathrm{O}(2)\times \mathrm{O}(2)$, and hence the Grassmannian can be expressed as homogeneous space
\begin{equation}
\textrm{Gr}(2,\mathbb{R}^4)=\frac{\mathrm{O}(4)}{\mathrm{O}(2)\times \mathrm{O}(2)}.
\end{equation}
In particular, $\dim\textrm{Gr}(2,\mathbb{R}^4)=6-2=4$. Let us select an explicit choice of fibration
\begin{equation}
p:\mathrm{O}(4)\rightarrow \textrm{Gr}(2,\mathbb{R}^4),
\end{equation}
by declaring that $p$ maps each rotation matrix to the span of its first two column vectors.

Denote by $U_\tau$ the open set of all $R\in  \mathrm{O}(4)$ satisfying $\sphericalangle(R(\Gamma),\Gamma)<\tfrac{1}{100}\rho(\tau)^{-3}$. Now, given any $\tau\ll 0$ and $R\in U_\tau$ we can write $R(\bar{M}_{\tau})$ as a graph of a function $u_R(x,\tau)$ over $\Gamma\cap B_{2\rho(\tau)}(0)$. Observe that the expression
\begin{equation}
\int_{\Gamma\cap B_{\rho(\tau)}(0)}e^{-|x|^2/4} u_R^2(x,\tau)\chi\!\left( \frac{r}{\rho(\tau)} \right),
\end{equation}
 as well as the relation
\begin{equation}\label{relation_der}
\int_{\Gamma\cap B_{\rho(\tau)}(0)}e^{-|x|^2/4} \langle Ax,\nu_{\Gamma} \rangle u_R(x,\tau)\chi\!\left( \frac{r}{\rho(\tau)} \right)=0\;\;\; \forall A\in o(4),
\end{equation}
is constant along the fibers of $p$. Set $V_\tau:=p(U_\tau)$, and observe that this is an open neighborhood of $[x_1x_2]\in\textrm{Gr}(2,\mathbb{R}^4)$. 

\begin{claim}\label{gras_prop}
Possibly after decreasing $U_\tau$, for every $\tau\ll 0$ there exists a unique $\pi(\tau)\in V_\tau$ such that  
\begin{equation}\label{opt_choice_eq}
\int_{\Gamma\cap B_{\rho(\tau)}(0)}e^{-|x|^2/4} \langle Ax,\nu_{\Gamma} \rangle u_{\tilde{\pi}(\tau)}(x,\tau)\chi\!\left( \frac{r}{\rho(\tau)} \right)=0\;\;\; \forall A\in \mathrm{o}(4),
\end{equation}
where $\tilde{\pi}(\tau)$ denotes some lift of $\pi(\tau)$ to $\mathrm{O}(4)$. Moreover, the function $\tau\mapsto \pi(\tau)$ is smooth. 
\end{claim}  

\begin{proof}[Proof of the claim]We will use the quantitative version of the implicit function theorem.\\

In order to first construct the required approximate solution, we 
consider the functional
\begin{equation}
H(\tau,\pi)=\frac{1}{2}\int_{\Gamma\cap B_{\rho(\tau)}(0)}e^{-|x|^2/4} u_\pi^2(x,\tau)\chi\!\left( \frac{r}{\rho(\tau)} \right). 
\end{equation}
Since $\bar{M}_{\tau}$ is a small graph over the cylinder with axis given by the 2-plane $[x_1x_2]\in\textrm{Gr}(2,\mathbb{R}^4)$, for every $\tau\ll 0$ there is a minimizer $\pi(\tau)$ for $H$ in a  small neighbourhood of $[x_1x_2]$, and we have $H(\tau,\pi(\tau))\leq C/\rho(\tau)^4$. Here, using \eqref{smaller_graph} it is not hard to see that the minimum is indeed attained (well inside) the open set $V_\tau$.\\

Fix some $\tau\ll 0$ and abbreviate $\tilde{\pi}:=\tilde{\pi}(\tau)$.
Now, let $R(\eta)\in \mathrm{SO}(4)$ be a one-parameter family of rotations with $R(0)=\tilde{\pi}$ and $R'(0)=A\tilde{\pi}$, where $A\in \mathrm{o}(4)$. Note that
\begin{equation}
R(\eta)^{-1}(x+u_{R(\eta)}(x,\tau)\nu_{\Gamma}(x))\in \bar{M}_\tau.
\end{equation}
Taking $\tfrac{d}{d\eta}|_{\eta=0}$ of this expression, observing the the resulting vector is of course orthogonal to $\nu(x+u_{\tilde \pi}\nu_\Gamma)$, and using the condition $\sphericalangle(\tilde{\pi}(\Gamma),\Gamma)<\tfrac{1}{100}\rho(\tau)^{-3}$, we infer that
\begin{equation}
\frac{d}{d\eta}|_{\eta=0} u_{R(\eta)}(x,\tau) = \langle Ax,\nu_{\Gamma}\rangle + O\left(\frac{1+|x|}{\rho(\tau)^2}\right).
\end{equation}
Using this, we obtain that the Euler-Lagrange equation for $H$ reads
\begin{align}\label{ELH}
\nabla_{A}H&=\int_{\Gamma\cap B_{\rho(\tau)}(0)}e^{-|x|^2/4} u_\pi(x,\tau)\langle Ax,\nu_{\Gamma} \rangle\chi\!\left( \frac{r}{\rho(\tau)} \right)+O\left(\frac{1}{\rho(\tau)^4}\right)\\
&=0 \;\;\; \forall A\in o(4).\nonumber
\end{align}
Thus, our minimizer $\pi(\tau)$ is an approximate solution of \eqref{opt_choice_eq}.\\

Now, to show existence and uniqueness of solutions to \eqref{opt_choice_eq}, as well as smooth dependence, let $A^1,\ldots A^4$ be the standard basis of
\begin{equation}
\mathcal{A}:=\{A\in so(4)\;|\; A_{12}=A_{34}=0\},
\end{equation}
and  define a map $F:\{(\tau,\pi)\, | \, \tau\leq \tau_0,\, \pi\in V_\tau\} \rightarrow \mathbb{R}^4$ by 
\begin{equation}
F(\tau,\pi)_i=\int_{\Gamma\cap B_{\rho(\tau)}(0)}e^{-|x|^2/4} \langle A^ix,\nu_{\Gamma} \rangle u_\pi(x,\tau)\chi\!\left( \frac{r}{\rho(\tau)} \right)
\end{equation}
for $i=1,\ldots,4$. Then, computing similarly as above we obtain
\begin{equation}\label{der_est}
\nabla_{A^i}F_j=\int_{\Gamma\cap B_{\rho(\tau)}(0)}e^{-|x|^2/4}\langle A^ix, \nu_{\Gamma} \rangle\langle A^jx, \nu_{\Gamma} \rangle\chi\!\left( \frac{r}{\rho(\tau)} \right)+O\left(\frac{1}{\rho(\tau)^{2}}\right).
\end{equation}
Since the only anti-symmetric matrix in $A\in \mathcal{A}$ such that $\langle Ax, \nu_{\Gamma} \rangle\equiv 0$ on $\Gamma$ is $0$, we see that the form $\nabla_{A^i}F_j$ is positive definite for $\tau\ll 0$, and so is in particular invertible.

Note that equation \eqref{der_est} holds in an $\rho(\tau)^{-3}/100$ neighborhood of $[x_1x_2]$ with uniform constants. Thus by the quantitative version of the inverse function theorem, in light of  \eqref{ELH} and \eqref{der_est},   there is a neighborhood of our fixed time $\tau$ such that the function $\pi(\cdot)$ can be chosen in it to satisfy \eqref{opt_choice_eq}. Finally, since by the above $F$ is uniformly bounded in spatial $C^1$, the size of that neighbourhood can be taken to be independent of $\tau$.  This finishes the proof of the claim. 
\end{proof}

Now, we can take a lift $S(\tau)$ of $\pi(\tau)$ to $\mathrm{SO}(4)$. Moreover, by post-composing with two $\mathrm{SO}(2)$ rotations, of the first two variables, and of the last two variables, we can further arrange that $A=S'S^{-1}$ satisfies $A_{12}=A_{34}=0$. This finishes the proof of the proposition.
\end{proof}

We now set
\begin{equation}
\tilde{M}_\tau^{X_0}=S^{X_0}(\tau)\bar M_\tau^{X_0},
\end{equation}
where $S^{X_0}(\tau)\in \mathrm{SO}(4)$ if from Proposition \ref{orthogonality} (Orthogonality), and set
\begin{equation}
u:=u_{S^{X_0}(\tau)},
\end{equation}
so $u$ describes $\tilde{M}_\tau^{X_0}$ as a graph over the cylinder $\Gamma$. Recall that we defined
\begin{equation}
\hat{u}(x,\tau)=u(x,\tau) \chi\left(\frac{r}{\rho(\tau)}\right).\\
\end{equation}

Our next task it to show that, despite of the truncation and the rotation, the function $\hat{u}$ evolves by the linearization of the rescaled MCF equation over the cylinder, up to a small error. Set 
\begin{equation}\label{r_grad}
e_r:=\nabla r=\frac{(x_1,x_2,0,0)}{r}.
\end{equation}

Let $\Delta_{\tau}=\Delta^{+}_{\tau}\cup \Delta^{-}_{\tau}$  be the region bounded by $\tilde{M}_{\tau}$ and $\Gamma$. Here, $\Delta^{+}_\tau$ denotes the region that is outside of $\Gamma$ and inside of $\tilde{M}_\tau$, and $\Delta^{-}_\tau$ denotes the region that is inside of $\Gamma$ outside of $\tilde{M}_\tau$.

\begin{proposition}[Gaussian area]\label{prop_calibration}
\label{calibraion}
For all $L \in [L_1,\rho(\tau)]$ and $\tau \ll 0$ we have the Gaussian area estimate
\begin{multline}
\int_{\tilde{M}_\tau \cap \{ r \geq L  \}} \, e^{-\frac{|x|^2}{4}} - \int_{\Gamma \cap \{r \geq L  \}} \, e^{-\frac{|x|^2}{4}}
\geq - \int_{\Delta_\tau\cap \{r=L\} }  \, e^{-\frac{|x|^2}{4}}\, | \langle e_r , \nu_{\mathrm{fol}}\rangle |.
\end{multline}

\end{proposition}

\begin{proof}
Let $\tilde{M}^{+}_\tau$ (respectively $\tilde{M}^{-}_\tau$) be the part of $\tilde{M}_{\tau}$ that lies outside (respectively inside) the cylinder. As above, let $\Delta^{\pm}_{\tau}$ be the corresponding parts of $\Delta_{\tau}$, and let $\Gamma^{\pm}=\Delta^{\pm}_{\tau}\cap \Gamma$.\\
Considering the foliation of $\Omega$ from Lemma \ref{foli_lemma},
since $\mathrm{div}(e^{-\frac{|x|^2}{4}} \, \nu_{\text{\rm fol}}) \geq  0$ in $\Delta_{\tau}^{+}\cap\{ r\geq L \}$, the divergence theorem yields that for every $R>L$ we have
\begin{multline} 
\int_{\tilde{M}^{+}_\tau \cap \{L \leq r \leq R\}} e^{-\frac{|x|^2}{4}} \, \langle \nu,\nu_{\text{\rm fol}} \rangle - \int_{\Gamma^{+} \cap \{L \leq r \leq R\}} e^{-\frac{|x|^2}{4}} \\ 
\geq -\int_{\Delta^{+}_\tau \cap \{r=L\}} e^{-\frac{|x|^2}{4}} \, |\langle e_r,\nu_{\text{\rm fol}} \rangle| - \int_{\Delta^{+}_\tau \cap \{r=R\}} e^{-\frac{|x|^2}{4}} \, |\langle e_r,\nu_{\text{\rm fol}} \rangle|. 
\end{multline} 
Similarly, since $\mathrm{div}(e^{-\frac{|x|^2}{4}} \, \nu_{\text{\rm fol}}) \leq   0$ in $\Delta_{\tau}^{-}\cap\{r\geq L\}$, the divergence theorem yields
\begin{multline} 
\int_{\Gamma^{-} \cap \{L \leq r \leq R\}} e^{-\frac{|x|^2}{4}}-\int_{\tilde{M}^{-}_\tau \cap \{L \leq r \leq R\}} e^{-\frac{|x|^2}{4}} \, \langle \nu,\nu_{\text{\rm fol}} \rangle  \\ 
\leq \int_{\Delta^{-}_\tau \cap \{r=L\}} e^{-\frac{|x|^2}{4}} \, |\langle e_r,\nu_{\text{\rm fol}} \rangle| + \int_{\Delta^{-}_\tau \cap \{r=R\}} e^{-\frac{|x|^2}{4}} \, |\langle e_r,\nu_{\text{\rm fol}} \rangle|. 
\end{multline}
Putting these together, we get 
\begin{multline} 
\int_{\tilde{M}_\tau \cap \{L \leq r \leq R\}} e^{-\frac{|x|^2}{4}} \, \langle \nu,\nu_{\text{\rm fol}} \rangle - \int_{\Gamma \cap \{L \leq r \leq R\}} e^{-\frac{|x|^2}{4}} \\ 
\geq -\int_{\Delta_\tau \cap \{r=L\}} e^{-\frac{|x|^2}{4}} \, |\langle e_r,\nu_{\text{\rm fol}} \rangle| - \int_{\Delta_\tau \cap \{r=R\}} e^{-\frac{|x|^2}{4}} \, |\langle e_r,\nu_{\text{\rm fol}} \rangle|. 
\end{multline} 
Using $|\langle \nu,\nu_{\text{\rm fol}} \rangle | \leq 1$ and $|\Delta_\tau \cap \{r=R\}|\leq CR^3$  and passing  $R \to \infty$ we conclude that
\begin{equation}
\int_{\tilde{M}_\tau \cap \{r \geq L\}} e^{-\frac{|x|^2}{4}} - \int_{\Gamma \cap \{r \geq L\}} e^{-\frac{|x|^2}{4}} \geq -\int_{\Delta_\tau \cap \{r=L\}} e^{-\frac{|x|^2}{4}} \, |\langle e_r,\nu_{\text{\rm fol}} \rangle|.
\end{equation}
This proves the proposition.
\end{proof}

Next, we have an inverse Poincare inequality:

\begin{proposition}[Inverse Poincare inequality]\label{Gaussian density analysis}
The graph function $u$ satisfies the integral estimates
\begin{equation}
\int_{\Gamma \cap \{|r| \leq L\}} e^{-\frac{|x|^2}{4}} \, |\nabla u(x,\tau)|^2 \leq  C \int_{\Gamma \cap \{|r| \leq \frac{L}{2}\}} e^{-\frac{|x|^2}{4}} \, u(x,\tau)^2
\end{equation}
and 
\begin{equation}
\int_{\Gamma \cap \{\frac{L}{2} \leq |r| \leq L\}} e^{-\frac{|x|^2}{4}} \, u(x,\tau)^2 \leq \frac{C}{L^2} \int_{\Gamma \cap \{|r| \leq \frac{L}{2}\}} e^{-\frac{|x|^2}{4}} \, u(x,\tau)^2
\end{equation}
for all $L \in [L_1,\rho(\tau)]$ and $\tau \ll 0$, where $C<\infty$ is a numerical constant.
\end{proposition}

\begin{proof}
The proof is quite similar to the one of Proposition 2.3 in \cite{BC}. Since $|\langle e_r,\nu_{\text{\rm fol}} \rangle| \leq C L^{-1} \, |x_3^2+x_4^2-2|$ by Lemma 4.11 in \cite{ADS}, we infer that
\begin{align} 
\int_{\Delta_\tau \cap \{r=L\}} e^{-\frac{|x|^2}{4}} \, |\langle e_r,\nu_{\text{\rm fol}} \rangle| 
&\leq   C L^{-1} \int_{\Gamma \cap \{r=L\}} e^{-\frac{|x|^2}{4}} \, u^2. 
\end{align}
Thus, Proposition \ref{prop_calibration} (Gaussian area) implies 
\begin{equation}
\int_{\tilde{M}_\tau \cap \{r \geq L\}} e^{-\frac{|x|^2}{4}} - \int_{\Gamma \cap \{r \geq L\}} e^{-\frac{|x|^2}{4}} \geq -C L^{-1} \int_{\Gamma \cap \{r=L\}} e^{-\frac{|x|^2}{4}} \, u^2 \,.
\end{equation}
On the other hand, we have
\begin{multline} 
\int_{\tilde{M}_\tau \cap \{r \leq L\}} e^{-\frac{|x|^2}{4}} - \int_{\Gamma \cap \{r \leq L\}} e^{-\frac{|x|^2}{4}} \\
\geq \int_{\{r\leq L\}} \int_0^{2\pi} e^{-\frac{r^2}{4}} \, \Big [ -C u^2 + \frac{1}{C} \, |\nabla^\Gamma u|^2 \Big ] \, d\theta \, dA,
\end{multline} 
where $C<\infty$ is a numerical constant. Hence, we can do the same computation as in the proof of Proposition 2.3 in \cite{BC} to obtain the desired result.
\end{proof}

Recall that $\tilde{M}_\tau$ is expressed as graph of a function $u(x,\tau)$ over $\Gamma\cap B_{2\rho(\tau)}(0)$ satisfying the estimate \eqref{small_graph}. Using that $\bar{M}_\tau=S(\tau)^{-1}\tilde{M}_\tau$ moves by rescaled mean curvature flow one obtains:

\begin{lemma}[evolution of graph function]\label{Error u-PDE}
The function $u(x,\tau)$ satisfies 
\begin{equation}
\partial_\tau u = \mathcal{L} u  + E+\langle A(\tau)x,\nu_\Gamma\rangle,
\end{equation}
where $A=S'S^{-1}$ and $\mathcal L$ is the linear operator on $\Gamma$ defined by
\begin{equation}\label{def_oper_ell}
\mathcal{L} f = \Delta f - \frac{1}{2} \, \langle x^{\text{\rm tan}},\nabla f \rangle + f,
\end{equation}
and where the error term can be estimated by
\begin{equation}\label{error-C0 norm}
|E| \leq C\rho(\tau)^{-1}( |u| + |\nabla u|+|A(\tau)|)
\end{equation}
for $\tau \ll 0$.
\end{lemma}

\begin{proof}
The proof is identical to the proof of \cite[Lemma 2.4]{BC}.
\end{proof}

Denote by $\mathcal{H}$ the Hilbert space of all functions $f$ on $\Gamma$ such that 
\begin{equation}\label{def_norm}
\|f\|_{\mathcal{H}}^2 = \int_\Gamma \frac{1}{(4\pi)^{3/2}} e^{-\frac{|x|^2}{4}} \, f^2 < \infty.
\end{equation}
 
\begin{proposition}[evolution of truncated graph function]
\label{Error hat.u-PDE 1}
The truncated function $\hat{u}(x,\tau) = u(x,\tau) \, \chi \big ( \frac{r}{\rho(\tau)} \big )$ satisfies 
\begin{equation}
\partial_\tau \hat{u} = \mathcal{L} \hat{u}  + \hat{E}+\langle A(\tau)x,\nu_\Gamma\rangle \chi\!\left(\frac{r}{\rho(\tau)}\right),
\end{equation}
where
\begin{equation}
\|\hat{E}\|_{\mathcal{H}} \leq C\rho^{-1} \big( \|\hat{u}\|_{\mathcal{H}}+|A(\tau)|\big)
\end{equation}
for $\tau \ll 0$.
\end{proposition} 

\begin{proof} 
We compute 
\begin{equation}
\partial_\tau \hat{u} = \mathcal{L} \hat{u} + \hat{E} + \langle A(\tau)x,\nu_\Gamma \rangle \, \chi \Big ( \frac{r}{\rho(\tau)} \Big ),
\end{equation}
where 
\begin{multline} 
\hat{E} = E \, \chi\! \Big ( \frac{r}{\rho(\tau)} \Big ) - \frac{2}{\rho(\tau)} \, \frac{\partial u}{\partial r} \, \chi' \Big ( \frac{r}{\rho(\tau)} \Big ) - \frac{1}{\rho(\tau)^2} \, u \, \chi'' \Big ( \frac{r}{\rho(\tau)} \Big ) \\ 
+ \frac{r}{2\rho(\tau)} \, u \, \chi' \Big ( \frac{r}{\rho(\tau)} \Big ) - \frac{r \rho'(\tau)}{\rho(\tau)^2} \, u \, \chi' \Big ( \frac{r}{\rho(\tau)} \Big ). 
\end{multline} 
Using Lemma \ref{Error u-PDE} (evolution of graph function), we deduce that 
\[|\hat{E}| \leq C\rho(\tau)^{-1}( |u| + |\nabla^\Gamma u| +  |A(\tau)|)\] 
for $r \leq \frac{\rho(\tau)}{2}$. Moreover, since $|\rho'(\tau)| \leq \rho(\tau)$ and $\rho(\tau) \to \infty$, we obtain 
\[|\hat{E}| \leq C |u| + C \rho^{-1}(|\nabla^\Gamma u| +  |A(\tau)|)\] 
for $\frac{\rho(\tau)}{2} \leq r \leq \rho(\tau)$. Thus, we can obtain the desired result as in the proof of \cite[Lemma 2.5]{BC}.
\end{proof}

\begin{lemma}[estimate for rotation function]\label{lemma_rotation_est}
The rotation can be estimated by
\begin{equation}
|A(\tau)|\leq C\rho^{-1}\|u\|_{\mathcal{H}}.
\end{equation}
In particular, we have
\begin{equation}\label{error_no_A}
\|\hat u_\tau -\mathcal{L} \hat u\|_{\mathcal{H}} \leq C\rho^{-1}\|\hat u\|_{\mathcal{H}}.
\end{equation}
\end{lemma}

\begin{proof}The proof is similar to the one of \cite[Lemma 2.6]{BC}. The conditions
$A_{12}(\tau)=A_{34}(\tau)=0$ imply
\begin{equation}\label{pos_def_form}
|A(\tau)|^2  \leq C \int_\Gamma e^{-\frac{|x|^2}{4}} \, \langle A(\tau)x,\nu_\Gamma \rangle^2 \, \chi\Big ( \frac{r}{\rho(\tau)} \Big ).
\end{equation}
Indeed, since non zero anti-symmetric matrices $A$ with $A_{12}=A_{34}=0$ are the velocity fields of rotations which change the axis of the cylinder, the bilinear-form
\begin{equation}
(A,B):=\int_\Gamma e^{-\frac{|x|^2}{4}} \, \langle Ax,\nu_\Gamma \rangle \langle Bx,\nu_\Gamma \rangle \, \chi\Big ( \frac{r}{\rho(\tau)} \Big ),
\end{equation}
is uniformly positive definite on such matrices. \\
Proposition \ref{orthogonality} (orthogonality) implies that $\hat{u}$, hence also $\partial_{\tau} \hat{u}$, is orthogonal to all $\langle Ax,\nu \rangle$ for each anti-symmetric $A$. Also $\mathcal{L}\hat{u}$ is orthogonal to $\langle Ax,\nu \rangle$, as the latter is in the kernel of $\mathcal{L}$. Thus, $\langle Ax,\nu \rangle$ is orthogonal to $\partial_{\tau}\hat{u}-\mathcal{L}\hat{u}=\hat{E} + \langle A(\tau)x,\nu_\Gamma \rangle \, \chi \big ( \frac{r}{\rho(\tau)} \big )$. From this and \eqref{pos_def_form}, we get

\begin{align*} 
|A(\tau)|^2 
&\leq C \int_\Gamma e^{-\frac{|x|^2}{4}} \, |\hat{E}| \, |\langle A(\tau)x,\nu_\Gamma \rangle|  \\ 
&\leq C \, \|\hat{E}\|_{\mathcal{H}} \, |A(\tau)| \\ 
&\leq C\rho^{-1} \, \|\hat{u}\|_{\mathcal{H}} \, |A(\tau)| + C\rho^{-1} \, |A(\tau)|^2, 
\end{align*} 
where in the last step we have used Proposition \ref{Error hat.u-PDE 1}. Consequently,
\begin{equation}
|A(\tau)| \leq C\rho^{-1} \, \|\hat{u}\|_{\mathcal{H}}.
\end{equation}
Using Proposition \ref{Error hat.u-PDE 1} once more, we get \eqref{error_no_A} as well. 
\end{proof}

The operator $\mathcal L$ is explicitly given by
\begin{equation}
\mathcal L=\frac{\partial^2}{\partial x_1^2} f+\frac{\partial^2}{\partial x_2^2} f + \frac{1}{2} \, \frac{\partial^2}{\partial \theta^2} f - \frac{1}{2} \, x_1 \, \frac{\partial}{\partial x_1} f- \frac{1}{2} \, x_2 \, \frac{\partial}{\partial x_2} f + f.\end{equation}
Analysing the spectrum of $\mathcal L$, the Hilbert space $\mathcal H$ from \eqref{def_norm} can be decomposed as
\begin{equation}
\mathcal H = \mathcal{H}_+\oplus \mathcal{H}_0\oplus \mathcal{H}_-,
\end{equation}
where 
\begin{align}
&\mathcal{H}_+ =\text{span}\{1,\cos\theta,\sin\theta,x_1,x_2\},\label{basis_hplus}\\
&\mathcal{H}_0 =\text{span}\{x_1\cos\theta,x_1\sin\theta,x_2\cos\theta,x_2\sin\theta,x_1^2-2,x_2^2-2,x_1x_2\}.
\end{align}
We have
\begin{align} 
&\langle \mathcal{L} f,f \rangle_{\mathcal{H}} \geq \frac{1}{2} \, \|f\|_{\mathcal{H}}^2 & \text{\rm for $f \in \mathcal{H}_+$,} \nonumber\\ 
&\langle \mathcal{L} f,f \rangle_{\mathcal{H}} = 0 & \text{\rm for $f \in \mathcal{H}_0$,} \\ 
&\langle \mathcal{L} f,f \rangle_{\mathcal{H}} \leq -\frac{1}{2} \, \|f\|_{\mathcal{H}}^2 & \text{\rm for $f \in \mathcal{H}_-$.} \nonumber
\end{align}
Consider the functions 
\begin{align}
&U_+(\tau) := \|P_+ \hat{u}(\cdot,\tau)\|_{\mathcal{H}}^2, \nonumber\\ 
&U_0(\tau) := \|P_0 \hat{u}(\cdot,\tau)\|_{\mathcal{H}}^2,\label{def_U_PNM} \\ 
&U_-(\tau) := \|P_- \hat{u}(\cdot,\tau)\|_{\mathcal{H}}^2, \nonumber
\end{align}
where $P_+, P_0, P_-$ denote the orthogonal projections to $\mathcal{H}_+,\mathcal{H}_0,\mathcal{H}_-$, respectively.

\begin{theorem}[Merle-Zaag alternative]\label{mz.ode.fine.neck}
For $\tau\to -\infty$ either the neutral mode is dominant, i.e.
\begin{equation}
U_-+U_+=o(U_0),
\end{equation}
or the unstable mode is dominant, i.e.
\begin{equation}
U_-+U_0\leq C\rho^{-1}U_+.
\end{equation}
\end{theorem}

\begin{proof}
Using Proposition \ref{Error hat.u-PDE 1} (evolution of truncated graph function) and Lemma \ref{lemma_rotation_est} (estimate for rotation function) we obtain
\begin{align} 
&\frac{d}{d\tau} U_+(\tau) \geq U_+(\tau) - C\rho^{-1} \, (U_+(\tau) + U_0(\tau) + U_-(\tau)), \nonumber\\ 
&\Big | \frac{d}{d\tau} U_0(\tau) \Big | \leq C\rho^{-1} \, (U_+(\tau) + U_0(\tau) + U_-(\tau)), \label{U_PNM_system}\\ 
&\frac{d}{d\tau} U_-(\tau) \leq -U_-(\tau) + C\rho^{-1} \, (U_+(\tau) + U_0(\tau) + U_-(\tau)). \nonumber
\end{align}
Hence, the Merle-Zaag ODE lemma (Lemma \ref{mz.ode}) implies the assertion.
\end{proof}

\bigskip

\section{Bubble-sheet analysis in the neutral mode}\label{sec_neutral_mode}

In this section, we prove Theorem \ref{main_thm} in the case where the neutral mode is dominant. Namely, throughout this section we consider a noncompact ancient noncollapsed flow in $\mathbb{R}^4$ whose truncated graph function $\hat u(\cdot,\tau)$ satisfies
\begin{equation}\label{neutral_dom}
U_-+U_+=o(U_0).
\end{equation}

The following lemma gives a rough  bound, showing that for $\tau\to -\infty$ the function $U_0$  decays slower than any exponential.

\begin{lemma}[rough decay estimate]\label{neutral mode lower bound}
For any $\delta>0$ we have
\begin{equation}
U_0(\tau) \geq e^{\delta\tau}
\end{equation}
for sufficiently large $-\tau$.
\end{lemma}

\begin{proof}
Given any $\delta>0$, the inequality \eqref{U_PNM_system} together with the assumption $U_-+U_+=o(U_0)$ implies that
\begin{equation}
\Big | \frac{d}{d\tau} U_0(\tau) \Big |\leq \tfrac{1}{2}\delta U_0
\end{equation}
for sufficiently large $-\tau$. Rewriting this as
\begin{equation}
\Big | \frac{d}{d\tau} \log U_0(\tau) \Big |\leq \tfrac{1}{2}\delta,
\end{equation}
integration gives $\log U_0(\tau) \geq -C+\frac{1}{2}\delta \tau$. This yields the desired result.
\end{proof}

\begin{proposition}\label{L^2 convergence lemma}
Every sequence $\{\tau_i\}$ converging to $-\infty$ has a subsequence $\{\tau_{i_m}\}$ such that 
\begin{equation}
\lim_{\tau_{i_m} \to -\infty}\frac{\hat u(\cdot,\tau_{i_m})}{\|\hat u(\cdot,\tau_{i_m})\|_{\mathcal{H}}}= q_{11}(x_1^2-2)+ q_{22}(x_2^2-2)+2q_{12}x_1x_2
\end{equation}
in $\mathcal{H}$-norm, where $\{q_{ij}\}$ is a nontrivial semi-negative definite $2\times2$-matrix.
\end{proposition}

\begin{proof}
By Proposition \ref{orthogonality} (orthogonality), we have
\begin{equation}
P_0\hat u \in \text{span}\{x_1^2-2,x_2^2-2,x_1x_2\}\subset \mathcal{H}_0.
\end{equation}
Therefore, every sequence $\tau_i \to -\infty$ has a subsequence $\{\tau_{i_m}\}$ such that 
\begin{equation}\label{L^2 convergence equation}
\lim_{\tau_{i_m} \to -\infty}\frac{\hat u(\cdot,\tau_{i_m})}{\|\hat u(\cdot,\tau_{i_m})\|_{\mathcal{H}}}=\mathcal{Q}(x_1,x_2),
\end{equation}
with respect to the ${\mathcal{H}}$-norm, where
\begin{equation}
\mathcal{Q}(x_1,x_2)=q_{11}(x_1^2-2)+ q_{22}(x_2^2-2)+2q_{12}x_1x_2,
\end{equation}
for some nontrivial matrix $Q=\{q_{ij}\}$. After an orthogonal change of coordinates in the $x_1x_2$-plane we can assume that $q_{12}=0$. Let us show that $q_{11} \leq 0$ (the same argument yields $q_{22}\leq 0$).\\
We denote by $\tilde{K}_\tau$ the region enclosed by $\tilde{M}_\tau$ and denote by $\mathcal{A}(x_1',x_2',\tau)$ the area of the cross section $\tilde{K}_\tau \cap \{(x_1,x_2)=(x_1',x_2')\}$. Explicitly, for $x_1^2+x_2^2\leq \rho(\tau)^2$ we have
\begin{align}\label{Area definition}
\mathcal{A}(x_1,x_2,\tau)&=\frac{1}{2}\int_0^{2\pi}\big(\sqrt{2}+u(\theta, x_1,x_2,\tau)\big)^2 d\theta\nonumber\\
&=2\pi +\sqrt{2} \int_0^{2\pi} u(\theta, x_1,x_2,\tau) d\theta + \tfrac12 \int_0^{2\pi} u(\theta, x_1,x_2,\tau)^2 d\theta\, .
\end{align}
By Brunn's concavity principle the function
\begin{equation}
(x_1,x_2)\mapsto\sqrt{\mathcal{A}(x_1,x_2,\tau)}
\end{equation}
is concave.  In particular, we have 
\begin{equation}
\sqrt{\mathcal{A}(x_1,x_2,\tau)}\geq \tfrac12 \sqrt{\mathcal{A}(x_1-2,x_2,\tau)}+\tfrac12 \sqrt{\mathcal{A}(x_1+2,x_2,\tau)}. 
\end{equation} 
This implies
 \begin{equation}\label{Brunn-Minkowski}
3\int_{-1}^{1}\int_{-1}^1\sqrt{\mathcal{A}}\, dx_2dx_1\geq \int_{-3}^{3}\int_{-1}^1\sqrt{\mathcal{A}}\,dx_2dx_1.
\end{equation}
Hence, for sufficiently large $-\tau_{i_m}$ combining \eqref{L^2 convergence equation}, \eqref{Area definition} and \eqref{Brunn-Minkowski} yields
 \begin{multline}\label{Mink112}
(3+o(1)) \|\hat u\|_{\mathcal{H}}\int_{-1}^{1}\int_{-1}^1\mathcal{Q}\, dx_2dx_1\\ 
\geq (1-o(1))\|\hat u\|_{\mathcal{H}}\int_{-3}^{3}\int_{-1}^1\mathcal{Q}\,dx_2dx_1-O(\|\hat u\|_{\mathcal{H}}^2).
\end{multline}
Indeed, given any compact intervals $I_1,I_2\subset\mathbb{R}$ setting $u_m=\hat{u}(\cdot,\tau_{i_m})$ we can compute
\begin{align}
\int_{I_1}\int_{I_2}\left(\sqrt{\mathcal{A}}-\sqrt{2\pi}\right) \, dx_2dx_1 &= \frac{1\pm o(1)}{2\sqrt{\pi}}\int_{I_1}\int_{I_2} \int_0^{2\pi} u_m \, d\theta dx_2 dx_1\nonumber\\
&=(1\pm o(1))\sqrt{\pi} \| u_m \|_{\mathcal{H}}\int_{I_1}\int_{I_2}\mathcal{Q}\, dx_2dx_1 \pm o(1)\| u_m \|_{\mathcal{H}},
\end{align}
which readily implies \eqref{Mink112}. Now, since $\|\hat u\|_{\mathcal{H}}>0$ and $\|\hat u\|_{\mathcal{H}}\to 0$, taking the limit as $m\rightarrow \infty$, we obtain
\begin{equation}
3\int_{-1}^{1}\int_{-1}^1\mathcal{Q}\, dx_2dx_1 \geq \int_{-3}^{3}\int_{-1}^1\mathcal{Q}\,dx_2dx_1.
\end{equation}
This implies
\begin{equation}
3q_{11}\int_{-1}^{1}\int_{-1}^1(x_1^2-2)\, dx_2dx_1 \geq q_{11}\int_{-3}^{3}\int_{-1}^1(x_1^2-2)\,dx_2dx_1.
\end{equation}
Since the integral of the left hand side is negative, while the integral on the right hand side is positive, we conclude that $q_{11} \leq 0$. This proves the proposition.
\end{proof}

\bigskip

Recall that in this section $M_t=\partial K_t$ denotes an ancient noncollapsed flow in $\mathbb{R}^4$ with dominant neutral mode, i.e. \eqref{neutral_dom} holds. Now, by the reduction from Section \ref{sec_coarse}, we can assume that its blowdown $\check{K}\equiv \check{K}_{t_0}$ is a convex cone that does not contain any line. Furthermore, rotating coordinates, we can arrange that  $\check{K}$ contains the positive $x_1$-axis, namely
\begin{equation}
\{ \lambda e_1\, |\, \lambda \geq 0\} \subseteq \check{K}.
\end{equation}
 
\begin{theorem}[blowdown in neutral mode]\label{thm_blowdown_neutral}
If $M_t=\partial K_t$ is as above, then its blowdown is a halfline, namely
\begin{equation}
\check{K}=\{ \lambda e_1\, |\, \lambda \geq 0\} .
\end{equation}
Moreover,
\begin{equation}
\lim_{\tau \to -\infty}\frac{\hat u(\cdot,\tau)}{\|\hat u(\cdot,\tau)\|_{\mathcal{H}}}= -c (x_2^2-2)
\end{equation}
in $\mathcal{H}$-norm, where $c=||x_2^2-2||^{-1}_{\mathcal{H}}$.
\end{theorem}

\begin{proof}By Proposition \ref{L^2 convergence lemma} given any sequence converging to $-\infty$ we can find a subsequence $\tau_m$ such that
\begin{equation}\label{subseq_conv}
\lim_{\tau_{m} \to -\infty}\frac{\hat u(\cdot,\tau_{i_m})}{\|\hat u(\cdot,\tau_{m})\|_{\mathcal{H}}}= q_{11}(x_1^2-2)+ q_{22}(x_2^2-2)+2q_{12}x_1x_2
\end{equation}
in $\mathcal{H}$-norm, where $Q=\{q_{ij}\}$ is a nontrivial semi-negative definite $2\times 2$-matrix. We will show that
\begin{equation}\label{contain_to_show}
\check{K}\subseteq \ker Q.
\end{equation}

To this end, observe that if $v$ is a unit vector in the $x_1x_2$-plane, denoting by $w$ the unit vector that is obtained from $v$ by a $\pi/2$-rotation, then
\begin{equation}\label{Q_1_ineq}  
\int_{0}^{1}\int_{0}^{1}\mathcal{Q}(rv+sw)drds = -\frac{5}{3}\mathrm{tr}(Q)+\frac{1}{2}w^T Q v \geq -\mathrm{tr}(Q) > 0.
\end{equation}
On the other hand, if $v\notin \ker Q$, we see that for sufficiently large $d=d(\sphericalangle(v,\ker Q))$, we have 
\begin{equation}\label{Q_2_ineq}  
\int_{0}^{1}\int_{d}^{d+1}\mathcal{Q}(rv+sw)drds\leq \frac{d^2}{2}v^TQv<0.
\end{equation}
Let $\tilde{K}_\tau$ be, as before,  the region enclosed by $\tilde{M}_\tau$.  Defining $\mathcal{A}$ as in \eqref{Area definition}, similarly as in the previous proof we have
\begin{equation}\label{lim_Q_2}
\int_{0}^{1}\int_{a}^{b}\frac{\sqrt{\mathcal{A}(rv+sw,\tau_m)}-\sqrt{2\pi}}{||\hat{u}||_{\mathcal{H}}}drds\rightarrow \sqrt{\pi} \int_{0}^{1}\int_{a}^{b}\mathcal{Q}(rv+sw)drds
\end{equation}
as $m\rightarrow \infty$. Combining \eqref{Q_1_ineq}, \eqref{Q_2_ineq} and \eqref{lim_Q_2} we see that for every $m$ sufficiently large, there exists $r_m,s_m\in [0,1]$ such that 
\begin{equation}\label{A_dec}
\mathcal{A}(r_mv+s_mw,\tau_m) >  \mathcal{A}((r_m+d)v+s_mw,\tau_m).\\
\end{equation}
Now, suppose towards a contradiction there is some $\omega\in \check K\setminus \ker Q$. Since $S(\tau)\rightarrow I$ as $\tau\rightarrow -\infty$, for all $-\tau$ sufficiently large we have
\begin{equation}
\sphericalangle (P_{12}(S(\tau)\omega),\ker Q)>\tfrac{1}{2}\sphericalangle(\omega,\ker Q),
\end{equation}
where $P_{12}$ denotes the projection to the $\mathrm{span}\{e_1,e_2\}$. Set
\begin{equation}
v_m:=\frac{P_{12}(S(\tau_m)\omega)}{|P_{12}(S(\tau_m)\omega)|}.
\end{equation}
Take $v=v_m$ in the previous discussion (and let $w_m$ be its $\pi/2$-rotation). It follows from Proposition \ref{blow_down_normals} (alternative description of blowdown) that for any $x\in\tilde{K}_\tau$ and $\omega\in \check{K}$ we have $x+\lambda S(\tau)\omega \in \tilde{K}_\tau$ for every $\lambda\in [0,\infty)$. Therefore, 
since $r_mv_m+s_mw_m\in \tilde{K}_{\tau_m}$, we see that 
\begin{equation}\label{stays_insisde}
r_mv_m+s_mw_m+\lambda S(\tau_m)\omega \in \tilde{K}_{\tau_m}\;\;\textrm{for every } \lambda \in [0,\infty).
 \end{equation}
On the other hand, thanks to Brunn's concavity principle, the function
\begin{equation}
r \mapsto \sqrt{\mathcal{A}(rv_m+s_m w_m,\tau)}
\end{equation}
is concave, for as long as it does not vanish. Together with \eqref{A_dec} this implies that for all $m$ sufficiently large, the area of the cross sections is decreasing for $r>r_m$, and vanishes at some finite $r_\ast$. This contradicts \eqref{stays_insisde}, as the ray would have nowhere to go. This proves \eqref{contain_to_show}.\\

Using the inclusion \eqref{contain_to_show}, since $\ker Q$ is $1$-dimensional, we infer that $\check K$ is $1$-dimensional. It follows that
\begin{equation}
\check{K}=\{ \lambda e_1\, |\, \lambda \geq 0\} ,
\end{equation}
and
\begin{equation}
\ker Q = \{ \textrm{$x_1$-axis}\}.
\end{equation}
Finally, since a normalized semi-negative $2\times 2$-matrix is uniquely characterized by its $1$-dimensional kernel, we see that subsequential convergence in \eqref{subseq_conv} in fact entails full convergence, and
\begin{equation}
\lim_{\tau \to -\infty}\frac{\hat u(\cdot,\tau)}{\|\hat u(\cdot,\tau)\|_{\mathcal{H}}}= -c (x_2^2-2),
\end{equation}
where $c=||x_2^2-2||_{\mathcal{H}}^{-1}$. This proves the theorem.
\end{proof}

\begin{corollary}[diameter of level sets]\label{diameter bound rough}
If $M_t=\partial K_t$ is as above, then given any $X$ and  $\delta>0$, assuming that the neutral mode dominates we have
\begin{equation*}
\bar M_\tau^X\cap \{x_1=0\} \subset B_{e^{-\delta\tau}}(0)
\end{equation*}
for sufficiently large $-\tau$.
\end{corollary}

\begin{proof}
By Theorem \ref{thm_blowdown_neutral} we have
\begin{equation}
\lim_{\tau \to -\infty}\frac{\hat u(\cdot,\tau)}{\|\hat u(\cdot,\tau)\|_{\mathcal{H}}}=-\frac{(x_2^2-2)}{||x_2^2-2||_{\mathcal{H}}}. 
\end{equation}
Hence, arguing similarly as above, for sufficiently large $-\tau$ we obtain
\begin{equation}
\int_{1}^{2}\int_{-1}^1\sqrt{\mathcal{A}} dx_2dx_1 - \sup_{s \in [-1,1]} \int_{1+s}^{2+s}\int_{9}^{11}\sqrt{\mathcal{A}} dx_2dx_1 \geq \big (c+o(1)\big) \|\hat u\|_{\mathcal{H}},
\end{equation}
where $c>0$ is a numerical constant. Moreover, by Lemma \ref{neutral mode lower bound} (rough decay estimate) we know that
\begin{equation}
\|\hat u\|_{\mathcal{H}}\geq e^{\delta \tau}.
\end{equation}
Therefore, there is some $a_+(\tau)\in (1,2)\times (-1,1)$, such that
\begin{equation}
\sqrt{\mathcal{A}(a_+(\tau),\tau)}-\sup_{s \in [-1,1]}\sqrt{\mathcal{A}(a_+(\tau)+se_1+10e_2,\tau)}\geq c e^{\delta\tau}.
\end{equation}
In particular, for all $\tau$ sufficiently negative this implies that
\begin{equation}
\sqrt{\mathcal{A}(a_+(\tau),\tau)}- \sqrt{\mathcal{A}(a_+(\tau)+10\langle e_2,S(\tau)e_2\rangle^{-1}S(\tau)e_2,\tau)}\geq c e^{\delta\tau}.
\end{equation} 
Hence, the concavity of $(x_1,x_2)\mapsto\sqrt{\mathcal{A}(x_1,x_2,\tau)}$ yields 
\begin{equation}
\mathcal{A}(a_+(\tau)+C e^{-\delta \tau}S(\tau)e_2,\tau)=0,
\end{equation}
for some constant $C<\infty$. Finally, recall that $\tilde{M}_\tau^X$ is a convex graph with respect to the height function $S(\tau)e_1$, and that the level sets of a convex graph monotonically increase. Observing also that $\langle a_+(\tau),S(\tau)e_1\rangle\geq 0$ we thus infer that
\begin{equation}\label{rough level set diameter}
\mathcal{A}(C e^{-\delta \tau}S(\tau)e_2,\tau)=0.
\end{equation}
Similarly, we can show $\mathcal{A}( -C e^{-\delta \tau}S(\tau)e_2,\tau)=0$. This proves the corollary.
\end{proof}

\bigskip

\section{Bubble-sheet analysis in the unstable mode}\label{sec_unstable}

In this section, we consider ancient noncollapsed flows with a bubble sheet tangent flow at $-\infty$, whose unstable mode is dominant. We will show that each such flow has some nonvanishing expansion parameters associated to it. This is the content of the  fine bubble-sheet theorem (Theorem \ref{thm Neck asymptotic}) and the nonvanishing expansion theorem (Theorem \ref{non.vanish.coeffs}).\\

More precisely, throughout  this section we assume the tilted rescaled flow $\tilde{M}^{X_0}_{\tau}$ around some point $X_0$, has a dominant unstable mode, i.e.
\begin{equation}\label{plus_dom_tilt}
U_0+U_{-} \leq C\rho^{-1} U_{+}.
\end{equation}

\bigskip
 
\subsection{Analysis of the untilted flow}  When dealing with the unstable mode case, it is convenient to work with the renormalized flow $\bar{M}^{X_0}_{\tau}$ itself, and not with its tilted version $\tilde{M}^{X_0}_{\tau}$. Our first task is to show that when dealing with the \textit{untilted} renormalized flow around \textit{any} point $X$, the unstable mode is still the dominant one. \\

Given any point $X$, for $\tau\ll 0$ (depending only on $Z(X)$) we can write $\bar M^X_\tau$ as a graph of a function $\bar u(\cdot,\tau)$ over $\Gamma \cap B_{2\rho(\tau)}(0)$, where $\rho=\rho^X(\cdot,\tau)$, which satisfies \eqref{univ_fns} and \eqref{small_graph}. We will work with the truncated function
\begin{equation}
\check u=\bar{u} \chi(r/\rho(\tau)).
\end{equation}

\begin{proposition}[evolution of truncated graph function]\label{refined error estimate} There exists some universal constant $C<\infty$ such that
\begin{equation}
\| (\partial_\tau - \mathcal{L})\check u\|_{\mathcal{H}}\leq C\rho^{-1}\|\check u\|_\mathcal{H}
\end{equation}
for $\tau\ll 0$.
\end{proposition}

\begin{proof} This follows from repeating the argument in Lemma \ref{Error u-PDE} and Proposition \ref{Error hat.u-PDE 1} with $A(\tau)=0$.
\end{proof}

Now, the functions 
\begin{align}
&\check U_+(\tau) := \|P_+ \check{u}(\cdot,\tau)\|_{\mathcal{H}}^2, \nonumber\\ 
&\check U_0(\tau) := \|P_0 \check{u}(\cdot,\tau)\|_{\mathcal{H}}^2,\\ 
&\check U_-(\tau) := \|P_- \check{u}(\cdot,\tau)\|_{\mathcal{H}}^2, \nonumber
\end{align}
satisfy the evolution inequalities \eqref{U_PNM_system}. Applying the Merle-Zaag ODE-Lemma (Lemma \ref{mz.ode}), we infer that there are universal constants $C_0,R<\infty$ (where we can assume that $R>10^3$) such that for every $X$ either 

\begin{equation}\label{vac_alter}
\check U_{+}+\check U_{-}=o(\check U_0),
\end{equation} 
or
\begin{equation}\label{plus_untilt_dom}
\check U_0+\check U_{-} \leq C_0\rho^{-1} \check U_{+}
\end{equation}
whenever $\rho\geq R$.

\begin{proposition}[dominant mode] The unstable mode is dominant for the untilted flow. More precisely, the inequality \eqref{plus_untilt_dom} holds for every $X$.
\end{proposition}

\begin{proof}
Let us first show that the statement hold for $X=X_0$. By assumption, the tilted flow with center $X_0$ has dominant unstable mode, i.e. 
\begin{equation}
U_0+U_{-} \leq C\rho^{-1} U_{+}.
\end{equation}
Together with the evolution inequalities \eqref{U_PNM_system} this also implies 
\begin{equation}\label{der_plus_tilt}
\frac{d}{d\tau} U_+(\tau) \geq (1-C\rho^{-1}) U_+.
\end{equation}
Note that \eqref{der_plus_tilt} implies
\begin{equation}\label{rough unstable dynamics}
\frac{d}{d\tau} e^{-\frac{9}{10}\tau} U_+(\tau) \geq 0,
\end{equation}
for sufficiently large $-\tau$. Hence,
\begin{equation}\label{U_decay}
\| u(\cdot,\tau)\|_{\mathcal{H}}^2=(1+o(1))U_+ \leq  Ce^{\frac{9}{10}\tau}.
\end{equation} 
Therefore, Lemma \ref{lemma_rotation_est} yields that 
\begin{equation}\label{S_decay}
|S^{X_0}(\tau)-I|\leq Ce^{\frac{9}{20}\tau}.
\end{equation}
Recall that 
\begin{equation}
\| u(\cdot,\tau)\|_{\mathcal{H}}^2=\frac{1}{(4\pi)^{3/2}}\int_\Gamma \hat{u}_{S^{X_0}(\tau)}(\cdot,\tau)^2 e^{-\frac{|x|^2}{4}},
\end{equation}
and note that for $X=X_0$ we have $\check{u}=\hat{u}_I$, hence
\begin{equation}
\| \check{u}(\cdot,\tau)\|_{\mathcal{H}}^2=\frac{1}{(4\pi)^{3/2}}\int_\Gamma \hat{u}_I(\cdot,\tau)^2 e^{-\frac{|x|^2}{4}}.
\end{equation}
Combining \eqref{U_decay} and \eqref{S_decay} we thus infer that the untilted flow satisfies
\begin{equation}\label{untilt_decay}
\| \check{u}(\cdot,\tau)\|_{\mathcal{H}}^2 \leq Ce^{\frac{9}{20}\tau}.
\end{equation}  
On the other hand, if we had $\check U_{+}+\check U_{-}=o(\check U_0)$, then arguing as in Lemma \ref{neutral mode lower bound}  we would see that $\| \check{u}(\cdot,\tau)\|_{\mathcal{H}}^2 \geq e^{\delta \tau}$ for every $\delta>0$ and $-\tau$ sufficiently large, which is inconsistent with \eqref{untilt_decay}. 
Thus, for $X=X_0$, we indeed get
 \begin{equation}
 \check U_0+\check U_{-} \leq C\rho^{-1} \check U_{+}.
 \end{equation}
Finally, any neck centered at a general point $X$ merges with the neck centered at $X_0$ as $\tau\to -\infty$. Thus, \eqref{plus_untilt_dom} holds for every $X$.
\end{proof}

Recapping, for any center $X$ we therefore have
\begin{equation}\label{plus_dom_untilt}
\check U_0+\check U_{-} \leq C_0\rho^{-1} \check U_{+},
\end{equation}
and, thanks to the evolution inequalities \eqref{U_PNM_system}, also
\begin{equation}\label{der_plus_untilt}
\frac{d}{d\tau} \check U_+ \geq (1 - C_0\rho^{-1}) \, \check U_+,
\end{equation}
whenever $\rho\geq R$, where $R\geq 10^3$ and $C_0$ are some universal constants. Increasing $R$ further, we can also assume that $R\geq 10C_0$.

\bigskip

\subsection{Graphical radius}  The goal of this section is to prove Proposition \ref{rough barrier}, which gives a lower bound for the optimal graphical radius.\\

We now fix some $\eps<1/R$ in the definition of the bubble-sheet scale (Definition \ref{eps_cyl}). In what follows, $C<\infty$ and $\mathcal{T}>-\infty$ will denote constants that are allowed to change from line to line, but depend \textit{only} on an upper bound of $Z(X)$. In particular, our initial choice of graphical radius satisfies $\rho^X(\tau)\geq R$ for $\tau\in (-\infty,\mathcal{T}]$ and \eqref{plus_dom_untilt} and \eqref{der_plus_untilt} hold for all $\tau\in (-\infty,\mathcal{T}]$.

\begin{lemma}[decay estimate]\label{refined estimate}
There exists some constant $C<\infty$, depending only on an upper bound for $Z(X)$, such that 
\begin{equation}\label{refined L^2 estimate}
\|\check u(\cdot,\tau)\|_{\mathcal{H}} \leq C e^{\frac{9}{20}\tau}
\end{equation}
and
\begin{equation}\label{refined C^10 estimate}
\|\check u(\cdot,\tau)\|_{C^{10}(\{r\leq 100\})}\leq C e^{\frac{9}{20}\tau}
\end{equation}
hold for $\tau \leq \mathcal{T} $.
\end{lemma}

\begin{proof}
Since $\bar M^X_{\mathcal{T}}$ is an $\varepsilon$-graph over $\Gamma\cap B_{2\rho^X}(0)$, we get $\|\check u(\cdot,\mathcal{T})\|_{\mathcal{H}}\leq   1 $.
Note that  \eqref{der_plus_untilt} implies
$
\tfrac{d}{d\tau} (e^{-\frac{9}{10}\tau}\check U_+) \geq 0
$
 for $\tau \leq \mathcal{T}$. Hence,
 \begin{equation}
 e^{-\frac{9}{10}
\tau}\check U_+ \leq e^{-\frac{9}{10}
\mathcal{T}}\|\check u(\cdot,\mathcal{T})\|_{\mathcal{H}}^2\leq C.
\end{equation}
This gives the first inequality \eqref{refined L^2 estimate}. The second inequality \eqref{refined C^10 estimate} follows from \eqref{refined L^2 estimate} by parabolic estimates. Indeed, one can first establish an $L^\infty$-bound via De Giorgi-Nash-Moser iteration (more precisely, one can apply Theorem \ref{Stampacchia estimate} with $k=0$), and then upgrade this to  a $C^{10}$-bound via standard Schauder theory.
\end{proof}

\bigskip

\begin{proposition}[improved graphical radius]\label{rough barrier}
There exists some $\mathcal{T}>-\infty$, depending only on an upper bound for $Z(X)$, such that 
\begin{equation}
\bar \rho(\tau)=e^{-\frac{1}{9}\tau}
\end{equation}
is a graphical radius function satisfying  \eqref{univ_fns} and \eqref{small_graph} for $\tau \leq \mathcal{T}$.
\end{proposition}

\begin{proof}
We recall from \cite[Thm. 8.2]{ADS} that the profile function $u_a$ of the ADS-barriers satisfies
\begin{equation}
u_a(x)  \geq \sqrt{2}\left(1-\frac{x^2}{a^2}\right)  
\end{equation}
on the interval $0\leq x\leq a$, and
\begin{equation}
 u_a(x) \leq \sqrt{2}\left(1-\frac{x^2-10}{1000a^2}\right)
\end{equation}
on the interval $8\leq x\leq 100$, provided that $a$ is large enough. 

The upper bound for $u_a$ and Lemma \ref{refined estimate} (decay estimate) imply that $\{r=10\}\cap \Gamma_a$ is enclosed by $\bar M^X_\tau$ for $a^2 \leq \frac{1}{C}e^{-\frac{9}{20}\tau}$. Since $\Gamma_a$ is enclosed by $\bar M^X_\tau$ for $\tau\ll0$, the maximum principle thus guarantees that
\begin{equation}
\textrm{$\bar M^X_\tau$ encloses $\{r\geq 10\}\cap \Gamma_a$ for $a^2 \leq Ce^{-\frac{9}{20}\tau}$.}
\end{equation}
Hence, by the estimate \eqref{refined C^10 estimate}, the convexity of $\bar{M}_\tau^X$, and the lower bound for $u_a$, we get that  $\bar{M}_\tau^X$ is graphical over the cylinder up to a radius $\frac{1}{C}e^{-\frac{9}{60}\tau}$, with $C^0$-norm less than $Ce^{\frac{9}{60}\tau}$.

We will now upgrade this $C^0$-estimate to a bound for the regularity scale by arguing similarly as in the proof of \cite[Proposition 4.10]{CHH}. Specifically, for any $X'=(x',t')\in\mathcal{M}$ denote by $R(X')$ the maximal radius $r$ such that $|A|\leq 1/r$ in the parabolic pall $P(X',r)$. Suppose towards a contradiction that $R(X')\ll \sqrt{-t'}$ for some $X'$ in the unrenormalized flow in the region under consideration. Consider Huisken's monotone quantity $\Theta_{X'}(r_j)$ from \cite{Huisken_monotonicity} centered at $X'$ at scales $r_j = 2^j R(X')$. Given any $\delta>0$ by quantitative differentiation \cite{CHN_stratification} there exists some $j\in \{1,\ldots, \lceil 2/\delta \rceil \}$ such that $\Theta_{X'}(r_{j+1})-\Theta_{X'}(r_{j-1})<\delta$, and consequently the flow must be $\eps$-close to a self-shrinker at scale $r_j$ centered at $X'$, where $\eps(\delta)\to 0$ for $\delta\to 0$. Fixing $\delta>0$ small enough we can arrange that $\eps\ll 1$ and $r_j \ll \sqrt{-t'}$. However, since in our noncollapsed setting the only self-shrinkers are round spheres, necks and bubble-sheets, this contradicts our graphical $C^0$-estimate.

Finally, interpolating the regularity scale estimate and the $C^0$-estimate, we see that up to a radius of $2e^{-\frac{1}{9}\tau}$, the hypersurface $\bar{M}_\tau^X$ is a $C^1$-graph over the cylinder with norm less than $\eps_0 e^{\frac{9}{60}\tau}$. Hence, using standard parabolic Schauder estimates for the renormalized flow around any point within that radius $e^{-\frac{1}{9}\tau}$ of the point $X$ yields the desired result.
\end{proof}

\bigskip

From now on, we work with $\rho=\bar{\rho}$ from Proposition \ref{rough barrier} (improved graphical radius), and in particular define $\check u, \check U_0, \check U_\pm, \ldots$ with respect to this improved graphical radius.  Note that the unstable mode is still dominant.

\bigskip

\begin{corollary}[sharp decay estimate]\label{sharp C^10 estimate}
There exist constants $C<\infty$ and $\mathcal{T}>-\infty$, depending only an upper bound for $Z(X)$, such that 
\begin{equation}
\|\check u\|_{\mathcal{H}} \leq Ce^{\tau/2},
\end{equation}
and
\begin{equation}\label{sharp_c10est}
\|\check u(\cdot,\tau)\|_{C^{10}(\{r\leq 100\})}\leq C e^{\tau/2}.
\end{equation}
holds for $\tau \leq \mathcal{T}$.
\end{corollary}

\begin{proof}
Combining the evolution inequality \eqref{der_plus_untilt}, Lemma \ref{refined estimate} (decay estimate) and Proposition \ref{rough barrier} (improved graphical radius) yields 
\begin{equation}
\frac{d}{d\tau}\left(e^{-\tau}\check U_+\right) \geq -Ce^{-\tau+\frac{1}{9}\tau+\frac{9}{10}\tau}=-Ce^{\frac{1}{90}\tau}
\end{equation}
for all $\tau\leq \mathcal{T}$.  Thus, $\check U_{+} \leq Ce^{\tau}$. Since the unstable mode is dominant, this proves that $\|\check u\|_{\mathcal{H}} \leq Ce^{\frac{\tau}{2}}$.\\
Moreover, Proposition \ref{rough barrier} (improved graphical radius), inequality \eqref{plus_dom_untilt} and Proposition \ref{refined error estimate} (evolution of truncated graph function) then  give
\begin{equation}\label{error_c10_prop}
\check U_0+\check U_- +\|(\partial_\tau-\mathcal{L})\check u\|_{\mathcal{H}}^2 \leq Ce^{\frac{10}{9}\tau}.
\end{equation}
Using this, \eqref{sharp_c10est} follows from parabolic estimates (see Theorem \ref{Stampacchia estimate} for details). This proves the corollary.
\end{proof}

\bigskip

\subsection{The fine bubble-sheet theorem} The goal of this section is to prove Theorem \ref{thm Neck asymptotic}. Recalling the basis of $\mathcal{H}_+$ from \eqref{basis_hplus} we can write
\begin{align}\label{def P_+check u}
P_+\check u^X=a_0^X(\tau)+a_1^X(\tau)x_1 +a_2^X(\tau)x_2+ a_3^X(\tau)\cos \theta + a_4^X(\tau)\sin \theta,
\end{align}
where the superscript  is to remind us that all these quantities can (a priori) depend on the center point $X$.

\begin{proposition}[estimate for coefficients]\label{est_coeffs}
The coefficients defined in \eqref{def P_+check u} satisfy the estimates
\begin{equation}\label{coeff_est1}
|a_0^X(\tau) |\leq Ce^{\tfrac{11}{18}\tau},
\end{equation}
and 
\begin{equation}\label{coeff_est2}
\sum_{i=1}^4\left|e^{-\tfrac{\tau}{2}} a_i^X(\tau)-\bar{a}_i^X\right| \leq Ce^{\tfrac{1}{9}\tau},
\end{equation}
for $\tau\leq \mathcal{T}$, where $\bar{a}_i^X$ are numbers that might depend on $X$.
\end{proposition}

\begin{proof}
Letting $\check E:=(\partial_\tau -\mathcal{L})\check u$, and using $\mathcal L \,1 = 1$, we compute
\begin{align}
\frac{d}{d\tau} a_0^X(\tau)
=(\tfrac{e}{2\pi})^{\frac{1}{4}}  \int (\mathcal{L}\check u+\check E) \tfrac{1}{4\pi} e^{-\frac{|x|^2}{4}}
=a_0^X(\tau)+(\tfrac{e}{2\pi})^{\frac{1}{4}}  \int   \tfrac{\check E}{4\pi}e^{-\frac{|x|^2}{4}}.
\end{align}
Hence, using Lemma \ref{refined error estimate}, Proposition \ref{rough barrier} and Corollary \ref{sharp C^10 estimate}, we obtain
\begin{align}
\Big|\frac{d}{d\tau}\left(e^{-\tau} a_0^X(\tau)\right) \Big| \leq Ce^{-\tau}\|\check E\|_{\mathcal{H}} \leq Ce^{-\tau+\frac{\tau}{2}+\frac{\tau}{9}}
\leq Ce^{-\frac{7}{18}\tau}.
\end{align}
Integrating this from $\tau$ to $\mathcal{T}$ implies \eqref{coeff_est1}.\\
In a similar manner, using $\mathcal{L} x_i=\frac{1}{2}x_i$ for $i=1,\ldots, 4$ we get
\begin{align}
\Big|\frac{d}{d\tau}\left(e^{-\frac{\tau}{2}} a_i^X(\tau)\right) \Big| \leq Ce^{-\frac{\tau}{2}}\|\check E\|_{\mathcal{H}} \leq Ce^{\frac{\tau}{9}},
\end{align}
so integrating from $-\infty$ to $\tau$ yields \eqref{coeff_est2} with
\begin{equation}
\bar{a}_i^X=\lim_{\tau\to -\infty} e^{-\tau/2}a_i^X(\tau).
\end{equation}
This proves the proposition.
\end{proof}

\bigskip

\begin{theorem}[The fine bubble-sheet theorem] \label{thm Neck asymptotic}Let $\{M_t\}$ be an ancient noncollapsed flow in $\mathbb{R}^4$, with a bubble sheet tangent flow at $-\infty$, whose unstable mode is dominant. Then
there exists some constants $ a_1,\cdots, a_4$, independent of the center point  $X=(x_0^1,x_0^2,x_0^3,x_0^4,t_0)$, such that 
\begin{align}\label{constants_univ_choice}
\bar a_1^X=a_1, &&\bar a_2^X= a_2, &&\bar a_3^X=a_3-x_0^3, &&\bar a_4^X=a_4-x_0^4.
\end{align}
Moreover, for every center point $X$ the truncated graph function $\check u^X(\cdot,\tau)$ of the renormalized flow $\bar{M}_\tau^X$ satisfies the estimates
\begin{equation}\label{main_thm_est1}
\left\| \check{u}^X-e^{\tfrac{\tau}{2}}\left(a_1 x_1 +a_2x_2 +\bar{a}_3^X\cos \theta + \bar{a}_4^X\sin \theta\right)\right\|_{\mathcal{H}}\leq C e^{\frac{5}{9}\tau},
\end{equation}
and
\begin{align}\label{main_thm_est2}
\left\| \check u^X-e^{\tfrac{\tau}{2}}\left(a_1 x_1 +a_2x_2 +\bar{a}_3^X\cos \theta + \bar{a}_4^X\sin \theta\right)\right\|_{L^{\infty}(\{r\leq 100\})} \leq C e^{\frac{19}{36}\tau}.
\end{align}
for $\tau\leq \mathcal{T}$, where $C<\infty$ and $\mathcal{T}>-\infty$ only depend on an upper bound on the bubble sheet scale $Z(X)$.
\end{theorem}

\begin{proof}
Consider the difference
\begin{equation}
D^X:=\check{u}^X-e^{\tau/2}\left(\bar{a}^X_1 x_1 +\bar{a}_2^X x_2 +\bar{a}_3^X\cos \theta + \bar{a}_4^X\sin \theta\right).
\end{equation}
Using  Proposition \ref{est_coeffs} (estimate for coefficients) we see that
\begin{equation}\label{point_ineq_D}
|D^X|\leq |\check{u}^X - P_{+}\check{u}^X| +  C (|x|+1) e^{\tfrac{11}{18}\tau}.
\end{equation}
Since by \eqref{error_c10_prop} we have $\check U_0+\check U_- \leq Ce^{\frac{10}{9}\tau}$, it follows that
\begin{equation}
\| D^X \|_{\mathcal{H}}\leq Ce^{\frac{5}{9}\tau},
\end{equation}
which proves \eqref{main_thm_est1} modulo the claim about the coefficients.\\

Next, we observe that 
\begin{align}
\|D^X \|_{L^2( \{ r\leq 100\})} \leq C\| D^X \|_{\mathcal{H}} \leq Ce^{\frac{5}{9}\tau},
\end{align}
and, using Corollary \ref{sharp C^10 estimate} (sharp decay estimate), that
\begin{align}
\|\nabla^3 D^X \|_{L^2( \{ r\leq 100\})} 
\leq C\| \hat{u}^X \|_{C^3( \{ r\leq 100\})}+Ce^{\frac{\tau}{2}}  \leq Ce^{\frac{\tau}{2}}.
\end{align}
Applying Agmon's inequality this yields
\begin{equation}\label{agmon_est}
\|D^X\|_{L^{\infty}( \{ r\leq 100\})}\leq C\|D^X\|_{L^2( \{ r\leq 100\})}^{\frac{1}{2}}\|D^X\|_{H^3( \{ r\leq 100\})}^{\frac{1}{2}}\leq C e^{\tfrac{19}{36}\tau}. 
\end{equation}

\bigskip

Finally, let us show that the parameters  $a_1,\cdots, a_4$, defined via \eqref{constants_univ_choice}, are independent of $X$.\\

First, let us show that they are independent of time translation. Denote by $\check{u}^{X'}$ the function obtained by considering the renormalized mean curvature flow with center $X'=(x_0^1,x_0^2,x_0^3,x_0^4,0)$. A direct calculation shows that for $\theta\in (0,2\pi]$, $x_1^2+x_2^2\leq 100$ we have
\begin{multline}
\check{u}^X\left(\frac{x_1}{\sqrt{1+t_0e^{\tau}}},\frac{x_2}{\sqrt{1+t_0e^{\tau}}},\theta,\tau-\log(1+t_0e^{\tau})\right)\\
=\frac{1}{\sqrt{1+t_0e^{\tau}}}(\sqrt{2}+\check{u}^{X'}(x_1,x_2,\theta,\tau))-\sqrt{2}.
\end{multline}
This implies 
\begin{equation}
\|\check{u}^X-\check{u}^{X'}\|_{L^{\infty}( \{ r\leq 100\})}=o(e^{\tau/2}),
\end{equation}
and thus together with \eqref{agmon_est} yields $\bar{a}_i^{X}=\bar{a}_i^{X'}$ for every $i=1,\ldots 4$.  \\

Now, comparing the renormalized flows with center $X'=(x_0^1,x_0^2,x_0^3,x_0^4,0)$ and center $X''=((0,0,0,0),0)$, we need to relate both the parameters of the functions $\check{u}^{X'}$ and $\check{u}^{X''}$ which describe the same point in the original flow, and the distance of such a point from the respective axii. This leads to 
\begin{multline}
\sqrt{2}+\check{u}^{X'}(x_1-x_0^1e^{\tau/2},x_2e-x_0^2e^{\tau/2},\theta+O(e^{\tau/2}),\tau)\\
=\mathrm{dist}\left(\left(\sqrt{2}+\check{u}^{X''}(x_1,x_2,\theta,\tau)\right)(\cos\theta,\sin\theta),e^{\tau/2}\left(x_0^3,x_0^4\right)\right).
\end{multline}
By Taylor expansion and Corollary \ref{sharp C^10 estimate} we have
\begin{multline}
\mathrm{dist}\left(\left(\sqrt{2}+\check{u}^{X''}(x_1,x_2,\theta,\tau)\right)(\cos\theta,\sin\theta),e^{\tau/2}\left(x_0^3,x_0^4\right)\right)\\
=\sqrt{2}+\check{u}^{X''}(x_1,x_2,\theta,\tau)-x_0^3\cos\theta e^{\tau/2}-x_0^4\sin\theta e^{\tau/2}+o(e^{\tau/2}).
\end{multline}
Together with \eqref{agmon_est}, the above formulas imply that
\begin{equation}
\bar{a}_1^{X'}=\bar{a}_1^{X''},\;\;\;\; \bar{a}_2^{X'}=\bar{a}_2^{X''},\;\;\;\; \bar{a}_3^{X'}=\bar{a}_3^{X''}-x_0^3,\; \;\;\;\;\bar{a}_4^{X'}=\bar{a}_4^{X''}-x_0^4.
\end{equation}
This finishes the proof of the theorem.
\end{proof}

\bigskip

\subsection{The nonvanishing expansion theorem} Our next goal it to show that $a_1$ and $a_2$ can not simultaneously vanish.\\

We decompose $\mathcal{H}_+$ into $\mathcal{H}_{1/2}=\text{span}\{x_1,x_2,\cos\theta,\sin\theta\}$ and $\mathcal{H}_1=\text{span}\{1\}$. Also, we define $P_{1/2}$ and $P_1$ as the projections to $\mathcal{H}_{1/2}$ and $\mathcal{H}_1$, respectively. In addition, we denote $\check U_{1/2}=\|P_{1/2}\check u\|^2_{\mathcal{H}}$ and $\check U_{1}=\|P_{1}\check u\|^2_{\mathcal{H}}$.

\begin{lemma}[decay if coefficients vanished]\label{lemma_dec_vanished}
If the center $X$ was such that $\bar{a}_1^X=\ldots = \bar{a}_4^X=0$, then there would be a constant $K_0>0$ such that
\begin{equation}\label{H-norm.limit.constant.mode}
\check U_1=K_0e^{2\tau}(1+O(e^{\frac{1}{9}\tau})),
\end{equation}
and moreover we would have
\begin{equation}\label{constant.dominant}
\check U_{1/2}+\check U_0+\check U_-\leq Ce^{\frac{1}{9}\tau}\check U_{1}.
\end{equation}
\end{lemma}

\begin{proof}
Assuming $\bar{a}_1^X=\ldots = \bar{a}_4^X=0$, Proposition \ref{est_coeffs} implies
\begin{equation}
a_0^X(\tau)^2+\ldots + a_4^X(\tau)^2\leq Ce^{\frac{11}{9}\tau},
\end{equation}
hence
\begin{equation}\label{est_1_12}
\check U_{1}+ \check U_{1/2}\leq e^{\frac{11}{9}\tau}.
\end{equation}
Moreover, since the unstable mode is dominant, we have
\begin{equation}\label{eq_since_unst}
\check U_0+\check U_-\leq Ce^{\frac{1}{9}\tau}(\check U_{1}+ \check U_{1/2}).
\end{equation}
Now, using Proposition \ref{refined error estimate} we get the evolution inequalities
\begin{align}
\big|\tfrac{d}{d\tau}\check U_{1/2}-\check U_{1/2}\big|&\leq Ce^{\frac{1}{9}\tau}(\check U_{1/2}+\check U_1),\label{U.system.1/2}\\
\big|\tfrac{d}{d\tau}\check U_{1}-2\check U_{1}\big|&\leq Ce^{\frac{1}{9}\tau}(\check U_{1/2}+\check U_1)\label{U.system.1}.
\end{align}
Applying the Merle-Zaag ODE-lemma (Lemma \ref{mz.ode}) with $U_0=e^{-\tau}\check U_{1/2}$, $U_{-}=0$, and $U_{+}=e^{-\tau}\check U_1$, we get either
$\check U_1=o(\check U_{1/2})$ or
$\check U_{1/2}\leq Ce^{\frac{1}{9}\tau}\check U_1$.
In the former case, arguing as in Lemma \ref{neutral mode lower bound} we could infer that $e^{-\tau}\hat U_{1/2}\geq e^{\frac{1}{9} \tau}$ for every $-\tau$ sufficiently large, contradicting \eqref{est_1_12}. Hence,
\begin{equation}\label{const_dom}
\check U_{1/2}\leq Ce^{\frac{1}{9}\tau}\check U_1.
\end{equation}
Together with \eqref{eq_since_unst} this proves the estimate \eqref{constant.dominant}.\\
Moreover, using \eqref{const_dom}, our differential inequality takes the form
\begin{equation}
\big|\tfrac{d}{d\tau}\check U_{1}-2\check U_{1}\big|\leq Ce^{\frac{1}{9}\tau} \check U_{1},
\end{equation}
which can be rewritten as
\begin{equation}
\big|\tfrac{d}{d\tau}(\log (e^{-2\tau}\check U_{1}))\big|\leq Ce^{\frac{1}{9}\tau}.
\end{equation}
Integrating this from $-\infty$ to $\tau$ gives that there exists some constant $\tilde{K}_0$ such 
\begin{equation}
|\log (e^{-2\tau}\check{U}_1)-\tilde{K}_0| \leq Ce^{\frac{1}{9}\tau} 
\end{equation}
for all $\tau \leq \mathcal{T}$. Exponentiating both sides and using the approximation $e^x\approx 1+x$ for small $x\in \mathbb{R}$ give \eqref{H-norm.limit.constant.mode}. In particular, $K_0=e^{\tilde{K}_0}>0$. 
\end{proof}

\bigskip

\begin{theorem}[the nonvanishing expansion theorem]\label{non.vanish.coeffs}
The coefficients from the fine-bubble sheet theorem satisfy $ |a_1|+ |a_2|>0$.
\end{theorem}

\begin{proof}
Suppose towards a contradiction that $a_1=a_2=0$. Then, we can choose a point $X$ such that $\bar{a}_1^X=\ldots = \bar{a}_4^X=0$. By Lemma \ref{lemma_dec_vanished} and parabolic estimates (see Theorem \ref{Stampacchia estimate}) we get
\begin{align} 
\|u(\cdot,\tau)\|_{C^4(\{r\leq 100\})} \leq Ce^{ \tau},
\end{align}
and
\begin{align}
x_3^2+x_4^2=2(1+Ke^{\tau})+o(e^{\tau})
\end{align}
on $\{r\leq 100\}$, where $K=(2e/\pi)^{1/4}K_0^{1/2}$. Here, we have determined the constant $K$ using the identity
\begin{equation}
\|2^{-1/2}Ke^{\tau}\|_{\mathcal{H}}^2=K_0e^{2\tau}=\check U_1+O(e^{\frac{19}{9}\tau}),
\end{equation}
which holds since
\begin{equation}
\int_{\Sigma}e^{-\frac{|x|^2}{4}}= 2^{\frac{3}{2}}\pi e^{-\frac{1}{2}} \left(\int_{-\infty}^{\infty}e^{-\frac{z^2}{4}}dz\right)^2=(2\pi/e)^{\frac{1}{2}}(4\pi)^{\frac{3}{2}}.
\end{equation}
Hence, the rescaled flow with the center $X'=X+(0,K)$ satisfies
\begin{align}\label{vanish.C0.norm}
x_3^2+x_4^2 = 2+o(e^{\tau}),
\end{align}
uniformly on $\{r\leq 100\}$, see also \cite[Lemma 5.11]{ADS} for a more detailed explanation. Since we re-centered by shifting only in time direction, the new point $X'$ still satisfies $\bar{a}_1^{X'}=\ldots = \bar{a}_4^{X'}=0$, so 
Lemma \ref{lemma_dec_vanished} gives some $K_0'>0$ such that 
\begin{equation}
e^{-2\tau}\check U_1=K_0'+O(e^{\tfrac{1}{9}\tau}).
\end{equation}
 Since $K_0'\neq 0$, this contradicts \eqref{vanish.C0.norm}. This proves the theorem.
\end{proof}

\bigskip

\section{Conclusion in the unstable mode case}\label{sec_unstable_end}

The goal of this section is to prove the following theorem. 

\begin{theorem}[unstable mode]\label{ruling_out_unstable}
The only noncompact ancient noncollapsed flow in $\mathbb{R}^4$, with bubble-sheet tangent flow at $-\infty$, whose unstable mode is dominant, is $\mathbb{R}\times2d$-bowl.
\end{theorem}

\begin{proof}By the reduction from Section \ref{sec_coarse} it is enough to prove that if the unstable mode is dominant, then its blowdown contains a line.\\

So let $M_t=\partial K_t$ be a noncompact ancient noncollapsed flow in $\mathbb{R}^4$, with bubble-sheet tangent flow at $-\infty$, whose unstable mode is dominant, and suppose towards a contradiction that its blowdown $\check K$ does not contain a line.
Then the flow is strictly convex, and  $\check{K}$ is a halfline or a wedge of angle less than $\pi$ in $\mathbb{R}^2\times \{0\}$.
Choosing suitable coordinates  we can assume that $\check{K}$ is symmetric across the $x_1$-axis, and is contained in the half space $\{x_1\geq 0\}$. By translating, we may also assume that $0\in M_0$ is the point in $M_0$ with smallest $x_1$-value. This implies that for every $h>0$ there exist a unique  point $x_h^{\pm}\in M_0\cap \{x_1=h\}$ at which $x_2$ is maximized/minimized. \\

By the fine bubble-sheet theorem (Theorem \ref{thm Neck asymptotic}) and the non-vanishing expansion theorem  (Theorem \ref{non.vanish.coeffs}) there exists expansion parameters $a_1,a_2$ associated to our flow such that $|a_1|+|a_2|>0$. 

\begin{claim}[bubble-sheet scale]\label{cyl_scale_bd}
There exists some constant $C<\infty$ such that
\begin{equation}
\sup_h Z(x^{\pm}_h)\leq C.
\end{equation}
\end{claim}

\begin{proof}[Proof of the claim]
We will argue similarly as in the proofs of \cite[Proposition 5.8]{CHH} and \cite[Proposition 6.2]{CHHW}.\\
Suppose towards a contradiction that $Z(x^{\pm}_{h_i})\to \infty$ for some sequence $\{h_i\}$ with $\lim_{i\rightarrow \infty}h_i=\infty$. Let $\mathcal{M}^i$ be the sequence of flows obtained by shifting $x^{\pm}_{h_i}$ to the origin, and parabolically rescaling by $Z(x^{\pm}_{h_i})^{-1}$. By \cite[Thm. 1.14]{HaslhoferKleiner_meanconvex} we can pass to a subsequential limit $\mathcal{M}^\infty$, which is an ancient noncollapsed flow that is weakly convex and smooth until it becomes extinct. Note also that $\mathcal{M}^\infty$ has bubble-sheet tangent flow at $-\infty$.

We next observe that, $\mathcal{M}^\infty$ cannot be a round shrinking $\mathbb{R}^2\times S^1$. Indeed, if such a cylinder became extinct at time $0$ that would contradict the definition of the bubble-sheet scale, and if it became extinct at some later time that would contradict the fact that $M^\infty_0\cap (\mathbb{R}^2\times \{0\})$ is a strict subset of $\mathbb{R}^2\times \{0\}$ by construction.\\
Thus, by Theorem \ref{mz.ode.fine.neck} (Merle-Zaag alternative) for the flow $\mathcal{M}^\infty$ either the neutral mode is dominant or the unstable mode is dominant.
If the neutral mode is dominant, then for large $i$, this contradicts the fact that $\mathcal{M}^i$ has dominant unstable mode. Indeed, on the one hand by Lemma \ref{neutral mode lower bound} (rough decay estimate) and Theorem \ref{thm_blowdown_neutral} (blowdown in neutral mode) the hypersurfaces $\bar{M}^{\infty,0}_\tau$ have some definite inwards quadratic bending, but on the other hand by the fine-bubble sheet theorem (Theorem \ref{thm Neck asymptotic})  the hypersurfaces $\bar{M}^{i,0}_\tau$ converge exponentially fast to the bubble sheet $\Gamma$. Since $\bar{M}^{i,0}_\tau$ converges locally smoothly to $\bar{M}^{\infty,0}_\tau$ this gives the desired contradiction for $i$ large enough.
If the unstable mode is dominant, then by the fine-bubble sheet theorem (Theorem \ref{thm Neck asymptotic}) and the non-vanishing expansion theorem (Theorem \ref{non.vanish.coeffs}) the limit $\mathcal{M}^\infty$ has some expension parameters $a_1^\infty,a_2^\infty$ that do not vanish simultaneously. However, this contradicts the fact that the expansion parameters of $\mathcal{M}^i$ are obtained from the expansion parameters $(a_1,a_2)$ of $\mathcal{M}$ by scaling by  $Z(x^{+}_{h_i})^{-1}\to 0$. This concludes the proof of the claim.
\end{proof}

Continuing the proof of the theorem, let $h_i\to \infty$ and consider the sequence $\mathcal{M}^i:=\mathcal{M}-(x_{h_i}^{+},0)$, which is obtained by translating in space time without rescaling.
Taking a subsequential limit, Claim \ref{cyl_scale_bd} implies that this limit $\mathcal{M}^\infty$ is an ancient noncollapsed flow with a bubble-sheet tangent at $-\infty$. Moreover, arguing as in the proof of Claim \ref{cyl_scale_bd} we see that $\mathcal{M}^\infty$ has a dominant unstable mode, with the same expansion parameters $a_1,a_2$ as $\mathcal{M}$.\\

On the other hand, by the choice of $x_{h_i}^{+}$, we have
\begin{equation}
\tfrac{x^+_{h_i}}{\|x^+_{h_i}\|}\rightarrow w^+ \in \partial \check{K},
\end{equation}
with $\langle w^{+},e_2 \rangle\geq 0$. Thus, $\mathcal{M}^\infty$ splits off a line in the direction $w^+$. Therefore, by \cite{BC} the limit $\mathcal{M}^\infty$ is $\mathbb{R}$ times a $2$-dimensional bowl, where the $\mathbb{R}$-factor is in the direction $w^+$, and where the translation direction $v^+$ is the orthogonal complement of $w^{+}$ in $\mathbb{R}^2\times \{0\}$ with 
\begin{equation}\label{speed_right}
\langle v^{+},e_2\rangle<0.
\end{equation}  

As the expansion parameters of $\mathcal{M}^\infty$ are also $a_1,a_2$ we see by observation (or by the fine-neck theorem from \cite{CHH}) that 
\begin{equation}
\begin{pmatrix} a_1 \\ a_2\end{pmatrix}=\gamma v^{+}
\end{equation}  
for some $\gamma>0$. Combining this with \eqref{speed_right}, we get that $a_2<0$.

Arguing similarly using $x_{h_i}^{-}$ gives that $a_2>0$; a contradiction. This concludes the proof of the theorem.
\end{proof}
\bigskip
\bigskip

\appendix

\section{Local $L^\infty$-estimate}

In this appendix, we consider the renormalized mean curvature flow given (in some ball) as a graph over the cylinder $\Gamma=\mathbb{R}^{n-d} \times \, S^d(\sqrt{2d})$, namely our variables are $(y,\sqrt{2d}\,\omega)\in \Gamma$ where $y\in \mathbb{R}^{n-d}$ and $\omega \in S^d$. Let
\begin{align}
\mathcal{L}u & =\Delta_\Gamma u-\tfrac12 x^{\text{tau}}\cdot \nabla_\Gamma u + u=\rho^{-1} \text{div}(\rho \nabla u) +\tfrac{1}{2d}\Delta_{S^d}u+u,
\end{align}
where $\rho(y)$ is the Gaussian density given by
\begin{equation}
\rho(y)=(4\pi)^{-\frac{d+1}{2}}e^{-\frac{d}{2}}e^{-\frac{|y|^2}{4}}.
\end{equation}
Given $R\gg 1$, we let $\Gamma_R=\{(y,\omega)\in\Gamma:|y|\leq R\}$ and  $Q(R)=\Gamma_R \times [-R^2,0]$. We consider smooth solutions $u:Q(2R) \to \mathbb{R}$ to the equation
\begin{equation}\label{appendix u-equation}
u_\tau =\mathcal{L}u +E.
\end{equation}

\begin{theorem}[{c.f. \cite[Theorem 6.17]{Lieberman}}]\label{Stampacchia estimate}
Suppose that for some constants $C_0,k<\infty$, the error $E$ satisfies
\begin{equation}
|E| \leq C_0(|u|+|\nabla u|)+k.
\end{equation}
Then,
\begin{equation}
\sup_{Q(R)}|u| \leq C\Bigg[k+ \bigg(\int_{Q(2R)} u^2\,d\emph{vol}_\Gamma d\tau \bigg)^{\frac{1}{2}}\Bigg],
\end{equation}
where $C=C(C_0,R,n)<\infty$.
\end{theorem}

\begin{proof}
Instead of $u$, we consider the function $\bar u=\rho u$ which solves
\begin{equation}
\bar u_\tau ={\rm{div}} \big(\rho \nabla_\Gamma (\rho^{-1}\bar u)\big)+\rho E.
\end{equation}
Since this equation satisfies the conditions of \cite[Theorem 6.17]{Lieberman}, given $q>2$ we can obtain the following inequality as the proof in Lieberman\footnote{The book states a stronger inequality, which is wrong, but easily correctable. Here, we provide the necessary modification of the argument for convenience of the reader.}
\begin{equation}
\int_{\Omega}\bar u^{q-2}|\nabla \bar u|^2v^{\alpha q-n-2}\xi^2d{\rm{vol}_\Gamma}d\tau \leq \frac{Cq^2}{R^2}\int_{\Omega}\bar u^qv^{\alpha q-n-2}d{\rm{vol}_\Gamma}d\tau,
\end{equation}
where $\Omega=Q(2R)\cap \{\bar u \geq kR\}$, $\xi=(1-|x|^2/4R^2)_+(1+t/4R^2)_+$ is a cut-off function, $v=(1-kR/\bar u)_+\in [0,1)$, and $\alpha=(n+2)/2$. Thus, $h=\bar u v^\alpha$ satisfies
\begin{equation}
\int_{\Omega}|\nabla h|^2\leq Cq^4R^{-2}\int_{\Omega}h^2v^{-2}
\end{equation}
Hence, the Sobolev type inequality from Lemma \ref{Sobolev} below yields 
\begin{multline}
C q^4 \int_{\Omega}h^2v^{-2} \geq \bigg(\int_{\Omega}|\nabla h|^2+\int_{\Omega}|h|^2\bigg)+\int_{\Omega}|h|^2\\
 \geq \frac{1}{C}\bigg(\int_{\Omega} |\nabla h|^2+|h|^2\bigg)^{\frac{n}{n+2}}\bigg(\int_{\Omega}h^2\bigg)^{\frac{2}{n+2}} \geq \frac{1}{C}\bigg(\int_{\Omega}h^{\frac{2(n+2)}{n}}\bigg)^{\frac{2}{n+2}}
\end{multline}
 for some constant $C=C(n,R)<\infty$. Thus, setting $\kappa=(n+2)/n$, $w=\bar u v^\alpha$, and $d\mu=R\xi (1-kR/\bar u)_+^{-n-2}d{\rm{vol}_\Gamma} d\tau$ gives
 \begin{equation}
 \bigg(\int_\Omega w^{\kappa q}d\mu\bigg)^{1/\kappa q}\leq C^{1/q}q^{4/q} \bigg(\int_\Omega w^{  q}d\mu\bigg)^{1/ q}.
 \end{equation}
 Hence, iterating this process with $q=2\kappa^j$ for $j\in \mathbb{N}$ yields the result. 
\end{proof}

The following lemma has been used in lieu of \cite[Theorem 6.9]{Lieberman}:

\begin{lemma}\label{Sobolev}
Suppose that $u\in C^{\infty}(Q(R))$ is a non-negative function satisfying $u=0$ on $\partial \Gamma_R \times [-R^2,0]$. Then, there exists some constant $C=C(n,R)<\infty$ such that
\begin{align}
\int_{Q(R)} u^{\frac{2(n+2)}{n}} \leq C \bigg(\int_{Q(R)}u^2 \bigg)^{\frac{2}{n}}\bigg(\int_{Q(R)} |\nabla u|^2+u^2 \bigg).
\end{align}
\end{lemma}

\begin{proof}
The H\"older inequality yields
\begin{equation}
\int_Q u^{\frac{2(n+2)}{n}} \leq \bigg(\int_Q u^2\bigg)^{\frac{1}{n}}\bigg(\int_Q u^{\frac{2(n+1)}{n-1}}\bigg)^{\frac{n-1}{n}}.
\end{equation}
Applying the Michael-Simon inequality this implies
\begin{equation}
\int_Q u^{\frac{2(n+2)}{n}} \leq C\bigg(\int_Q u^2\bigg)^{\frac{1}{n}}\bigg(\int_Q u^{\frac{n+2}{n}}|Du|+u^{\frac{2(n+1)}{n}}\bigg)^{\frac{n-1}{n}}
\end{equation}
Using again H\"older's inequality we obtain the desired result.
\end{proof}

\bigskip

\section{Merle-Zaag ODE-lemma}

We recall the following variant of the Merle-Zaag ODE-lemma \cite[Lemma A.1]{MZ}, which has been proved in \cite[Lemma B.1]{ChoiMant}:

\begin{lemma} \label{mz.ode}
 $U_0$, $U_+$, $U_- : (-\infty, 0] \to \mathbb{R}$ are absolutely continuous non-negative functions satisfying $U_0+U_++U_->0$ and
	\begin{equation} 
		\liminf_{s \to -\infty} U_{-}(s) = 0.
	\end{equation}
Suppose that there exist some constant $c_0>0$ and positive increasing function $\sigma:(-\infty,0]\to \mathbb{R}$ such that $\lim_{s\to -\infty}\sigma(s)=0$ and the following hold
	\begin{align}
				|U_0'| &\leq \sigma (U_0 + U_- + U_+), \\
			U_-'  &\leq -c_0U_-+\sigma (U_0 + U_+), \\
			U_+' &\geq c_0U_+ - \sigma(U_0 + U_-).
\end{align}			
Then, there exists $c=c(c_0) > 0$ and $\delta=\delta(c_0)>0$ such that if $\sigma(s_0)<\delta$ then   
	\begin{equation} 
		U_- \leq c\sigma (U_0 + U_+) \text{ on } (-\infty, s_0],
	\end{equation}
and either
	\begin{equation} 
		U_+ \leq o(U_0),
	\end{equation}
	or 
	\begin{equation} 
	U_0 \leq c \sigma U_+ 
	\end{equation}
for all $s\leq s_0$.
\end{lemma}

\bigskip

\bibliography{wing}

\newcommand{\noopsort}[1]{} \newcommand{\singleletter}[1]{#1}
\begin{thebibliography}{CHHW22}

\bibitem[ADS19]{ADS}
S.~Angenent, P.~Daskalopoulos, and N.~Sesum.
\newblock Unique asymptotics of ancient convex mean curvature flow solutions.
\newblock {\em J. Differential Geom.}, 111(3):381--455, 2019.

\bibitem[ADS20]{ADS2}
S.~Angenent, P.~Daskalopoulos, and N.~Sesum.
\newblock Uniqueness of two-convex closed ancient solutions to the mean
  curvature flow.
\newblock {\em Ann. of Math. (2)}, 192(2):353--436, 2020.

\bibitem[And12]{Andrews_noncollapsing}
B.~Andrews.
\newblock Noncollapsing in mean-convex mean curvature flow.
\newblock {\em Geom. Topol.}, 16(3):1413--1418, 2012.

\bibitem[AW94]{AltschulerWu}
S.~Altschuler and L.~Wu.
\newblock Translating surfaces of the non-parametric mean curvature flow with
  prescribed contact angle.
\newblock {\em Calc. Var. Partial Differential Equations}, 2(1):101--111, 1994.

\bibitem[BC19]{BC}
S.~Brendle and K.~Choi.
\newblock Uniqueness of convex ancient solutions to mean curvature flow in
  $\mathbb{R}^3$.
\newblock {\em Invent. Math.}, 217(1):35--76, 2019.

\bibitem[BC21]{BC2}
S.~Brendle and K.~Choi.
\newblock Uniqueness of convex ancient solutions to mean curvature flow in
  higher dimensions.
\newblock {\em Geom. Topol.}, 25(5):2195--2234, 2021.

\bibitem[BH16]{BH_surgery}
S.~Brendle and G.~Huisken.
\newblock Mean curvature flow with surgery of mean convex surfaces in {$\Bbb
  R^3$}.
\newblock {\em Invent. Math.}, 203(2):615--654, 2016.

\bibitem[Bre15]{Brendle_inscribed}
S.~Brendle.
\newblock A sharp bound for the inscribed radius under mean curvature flow.
\newblock {\em Invent. Math.}, 202(1):217--237, 2015.

\bibitem[CHH22]{CHH}
K.~Choi, R.~Haslhofer, and O.~Hershkovits.
\newblock Ancient low-entropy flows, mean-convex neighborhoods, and uniqueness.
\newblock {\em Acta Math.}, 228(2):217--301, 2022.

\bibitem[CHHW22]{CHHW}
K.~Choi, R.~Haslhofer, O.~Hershkovits, and B.~White.
\newblock Ancient asymptotically cylindrical flows and applications.
\newblock {\em Invent. Math.}, 229(1):139--241, 2022.

\bibitem[CHN13]{CHN_stratification}
J.~Cheeger, R.~Haslhofer, and A.~Naber.
\newblock Quantitative stratification and the regularity of mean curvature
  flow.
\newblock {\em Geom. Funct. Anal.}, 23(3):828--847, 2013.

\bibitem[CM12]{CM_generic}
T.~Colding and W.~Minicozzi.
\newblock Generic mean curvature flow {I}; generic singularities.
\newblock {\em Ann. of Math. (2)}, 175(2):755--833, 2012.

\bibitem[CM15]{CM_uniqueness}
T.~Colding and W.~Minicozzi.
\newblock Uniqueness of blowups and {L}ojasiewicz inequalities.
\newblock {\em Ann. of Math. (2)}, 182(1):221--285, 2015.

\bibitem[CM16]{CM_singular_set}
T.~Colding and W.~Minicozzi.
\newblock The singular set of mean curvature flow with generic singularities.
\newblock {\em Invent. Math.}, 204(2):443--471, 2016.

\bibitem[CM22]{ChoiMant}
K.~Choi and C.~Mantoulidis.
\newblock Ancient gradient flows of elliptic functionals and {M}orse index.
\newblock {\em Amer. J. Math.}, 144(2):541--573, 2022.

\bibitem[HIMW19]{HIMW}
D.~Hoffman, T.~Ilmanen, F.~Martin, and B.~White.
\newblock Graphical translators for mean curvature flow.
\newblock {\em Calc. Var. Partial Differential Equations}, 58(4):Paper No. 117,
  29, 2019.

\bibitem[HK15]{HaslhoferKleiner_inscribed}
R.~Haslhofer and B.~Kleiner.
\newblock On {B}rendle's estimate for the inscribed radius under mean curvature
  flow.
\newblock {\em Int. Math. Res. Not. IMRN}, (15):6558--6561, 2015.

\bibitem[HK17a]{HaslhoferKleiner_meanconvex}
R.~Haslhofer and B.~Kleiner.
\newblock Mean curvature flow of mean convex hypersurfaces.
\newblock {\em Comm. Pure Appl. Math.}, 70(3):511--546, 2017.

\bibitem[HK17b]{HaslhoferKleiner_surgery}
R.~Haslhofer and B.~Kleiner.
\newblock Mean curvature flow with surgery.
\newblock {\em Duke Math. J.}, 166(9):1591--1626, 2017.

\bibitem[HS09]{HuiskenSinestrari_surgery}
G.~Huisken and C.~Sinestrari.
\newblock Mean curvature flow with surgeries of two-convex hypersurfaces.
\newblock {\em Invent. Math.}, 175(1):137--221, 2009.

\bibitem[Hui90]{Huisken_monotonicity}
G.~Huisken.
\newblock Asymptotic behavior for singularities of the mean curvature flow.
\newblock {\em J. Differential Geom.}, 31(1):285--299, 1990.

\bibitem[Hui93]{Huisken_shrinker}
G.~Huisken.
\newblock Local and global behaviour of hypersurfaces moving by mean curvature.
\newblock In {\em Differential geometry: partial differential equations on
  manifolds ({L}os {A}ngeles, {CA}, 1990)}, volume~54 of {\em Proc. Sympos.
  Pure Math.}, pages 175--191. Amer. Math. Soc., Providence, RI, 1993.

\bibitem[Ilm95]{Ilmanen_monotonicity}
T.~Ilmanen.
\newblock Singularities of mean curvature flow of surfaces.
\newblock {\em https://people.math.ethz.ch/~ilmanen/papers/sing.ps}, 1995.

\bibitem[KM14]{KM}
S.~Kleene and N.~Moller.
\newblock Self-shrinkers with a rotational symmetry.
\newblock {\em Trans. Amer. Math. Soc.}, 366(8):3943--3963, 2014.

\bibitem[Lie96]{Lieberman}
G.~Lieberman.
\newblock {\em Second order parabolic differential equations}.
\newblock World Scientific Publishing Co., Inc., River Edge, NJ, 1996.

\bibitem[MZ98]{MZ}
F.~Merle and H.~Zaag.
\newblock Optimal estimates for blowup rate and behavior for nonlinear heat
  equations.
\newblock {\em Comm. Pure Appl. Math.}, 51(2):139--196, 1998.

\bibitem[SW09]{ShengWang}
W.~Sheng and X.~Wang.
\newblock Singularity profile in the mean curvature flow.
\newblock {\em Methods Appl. Anal.}, 16(2):139--155, 2009.

\bibitem[SX20]{SpruckXiao}
J.~Spruck and L.~Xiao.
\newblock Complete translating solitons to the mean curvature flow in {$\Bbb
  R^3$} with nonnegative mean curvature.
\newblock {\em Amer. J. Math.}, 142(3):993--1015, 2020.

\bibitem[Wan11]{Wang_convex}
X.~Wang.
\newblock Convex solutions to the mean curvature flow.
\newblock {\em Ann. of Math. (2)}, 173(3):1185--1239, 2011.

\bibitem[Whi00]{White_size}
B.~White.
\newblock The size of the singular set in mean curvature flow of mean convex
  sets.
\newblock {\em J. Amer. Math. Soc.}, 13(3):665--695, 2000.

\bibitem[Whi03]{White_nature}
B.~White.
\newblock The nature of singularities in mean curvature flow of mean-convex
  sets.
\newblock {\em J. Amer. Math. Soc.}, 16(1):123--138, 2003.

\end{thebibliography}

\bibliographystyle{alpha}

\vspace{10mm}

{\sc Kyeongsu Choi, School of Mathematics, Korea Institute for Advanced Study, 85 Hoegiro, Dongdaemun-gu, Seoul, 02455, South Korea}\\

{\sc Robert Haslhofer, Department of Mathematics, University of Toronto,  40 St George Street, Toronto, ON M5S 2E4, Canada}\\

{\sc Or Hershkovits, Institute of Mathematics, Hebrew University, Givat Ram, Jerusalem, 91904, Israel}\\

\emph{E-mail:} choiks@kias.re.kr, roberth@math.toronto.edu, or.hershkovits@mail.huji.ac.il

\end{document}